\documentclass{lmcs}
\pdfoutput=1

\usepackage{lastpage}
\lmcsdoi{22}{2}{36}
\lmcsheading{}{\pageref{LastPage}}{}{}%
{Feb.~20,~2025}{Jun.~30,~2026}{}

\keywords{String diagrams, categories for relations, gs-monoidal categories, restriction categories, Markov categories, semiring-weighted monads}

\usepackage[english]{babel}
\usepackage[all,2cell]{xy}

\usepackage[T1]{fontenc}

\usepackage{graphicx}%
\usepackage{multirow}%
\usepackage{amsmath,amssymb,amsfonts}%
\usepackage{mathrsfs}%
\usepackage[title]{appendix}%
\usepackage{xcolor}%
\usepackage{booktabs}%

\usepackage{url}
\usepackage{doi}
\usepackage{quiver}
\usepackage{mathtools}

\newcommand{\comment}[1]{ }
\usepackage{hyperref}
\usepackage{xcolor}

\newcommand{\cobang}{\mathrel{\rotatebox[origin=c]{180}{!}}}

\newcommand{\freccia}[3]{#2 \colon #1  \to #3}

\newcommand{\duefreccia}[3]{\xymatrix@C=0.5cm{#2 \colon #1  \ar@{=>}[r] &  #3}}

\newcommand{\comsquare}[8]{ \xymatrix@+1pc{ 
#1 \ar[r]^{#5} \ar[d]_{#6} & #2 \ar[d]^{#7} \\
#3 \ar[r]_{#8} & #4 
}}
\newcommand{\pullback}[8]{ \xymatrix@+1pc{ 
#1 \pullbackcorner \ar[r]^{#5} \ar[d]_{#6} & #2 \ar[d]^{#7} \\
#3 \ar[r]_{#8} & #4 
}}
\newcommand{\quadratocomm}[8]{ \xymatrix@+1pc{ 
#1 \ar[r]^{#5} \ar[d]_{#6} & #2 \ar[d]^{#7} \\
#3 \ar[r]_{#8} & #4 
}}
\newcommand{\comsquarelargo}[8]{ \xymatrix@+1pc{ 
#1 \ar[rr]^{#5} \ar[d]_{#6} && #2 \ar[d]^{#7} \\
#3 \ar[rr]_{#8} && #4 
}}
\newcommand{\parallelmorphisms}[4]{\xymatrix@+1pc{
#1 \ar @<+4pt>[r]^{#2} \ar @<-4pt>[r]_{#3} & #4
}}
\newcommand{\relation}[4]{\xymatrix@+1pc{
\angbr{#2}{#3}\colon #1 \ar @<+4pt>[r] \ar @<-4pt>[r] & #4
}}
\newcommand{\frecceparalleleopposte}[4]{\xymatrix@+1pc{
#1 \ar@<+4pt>[r]^{#2} \ar@<-4pt>@{<-}[r]_{#3} & #4
}}
\newcommand{\equalizer}[6]{\xymatrix@+1pc{
#1 \ar[r]^{#2} & #3 \ar @<+4pt>[r]^{#4} \ar @<-4pt>[r]_{#5} & #6
}}
\newcommand{\coequalizer}[6]{\xymatrix@+1pc{
 #1 \ar @<+4pt>[r]^{#2} \ar @<-4pt>[r]_{#3} & #4 \ar[r]^{#5} & #6
}}

\newcommand{\subobject}[3]{\xymatrix{
#1 \ar@{>->}[r]^{#2} & #3
}}

\newcommand{\pullbackcorner}[1][ul]{\save*!/#1+1.2pc/#1:(1,-1)@^{|-}\restore}

\def\mA{\mathcal{A}}

\def\mC{\mathcal{C}}

\def\mD{\mathcal{D}}

\def\Rel{\mathbf{Rel}}

\def\id{\operatorname{ id}}         

\newcommand{\angbr}[2]{\langle #1,#2 \rangle}

\usepackage{quiver}
\usepackage{mathtools}

\usepackage{tikz}
\usepackage{tikzit}

\usetikzlibrary{shapes}

\usepackage{freetikz}
\usetikzlibrary{decorations.markings,positioning,patterns}
\usetikzlibrary{shadows}
\usepackage{array}
\usepackage{tikz-cd}
\usepackage[utf8]{inputenc}

\tikzstyle{nodonero}=[fill=black, draw=black, shape=circle]
\tikzstyle{box}=[fill=white, draw=black, shape=rectangle]
\tikzstyle{medium box}=[fill=white, draw=black, shape=rectangle, minimum width=0.7cm, minimum height=0.7cm]
\tikzstyle{bn}=[fill=black, draw=black, shape=circle, inner sep=1.5pt]
\tikzstyle{state}=[fill=white, draw=black, regular polygon, regular polygon sides=3, minimum width=0.8cm, shape border rotate=180, inner sep=0pt]
\tikzstyle{costate}=[fill=white, draw=black, regular polygon, regular polygon sides=3, minimum width=0.8cm, inner sep=0pt]
\tikzstyle{comp}=[fill={rgb,255: red,191; green,0; blue,64}, draw={rgb,255: red,191; green,0; blue,64}, shape=circle, inner sep=1.5pt]

\tikzstyle{ds}=[-, dashed, dash pattern=on 1mm off 1mm]

\usepackage[all,2cell]{xy}
\UseAllTwocells
\xyoption{v2}

\title{A taxonomy of categories for relations}
\author{Cipriano Junior Cioffo\lmcsorcid{0000-0002-4189-0930}}[a]
\author{Fabio Gadducci\lmcsorcid{0000-0003-0690-3051}}[a]
\address{Department of Computer Science, University of Pisa, Pisa, Italy}
\email{ciprianojunior.cioffo@di.unipi.it, fabio.gadducci@unipi.it}

\author{Davide Trotta \lmcsorcid{0000-0003-4509-594X}}[b]
\address{Department of Mathematics, University of Padova, Padova, Italy}
\email{trottadavide92@gmail.com}

\newenvironment{mytheorem}{\begin{thm}}{\end{thm}}
\newenvironment{mylemma}{\begin{lem}}{\end{lem}}
\newenvironment{mycorollary}{\begin{cor}}{\end{cor}}
\newenvironment{myproposition}{\begin{prop}}{\end{prop}}

\newenvironment{mydefinition}{\begin{defi}}{\end{defi}}

\newenvironment{myremark}{\begin{rem}}{\end{rem}}
\newenvironment{myexample}{\begin{exa}}{\end{exa}}

\begin{document}
\maketitle

\begin{abstract}
The study of categories that abstract the structural properties of relations has been extensively developed over the years, resulting in a rich and diverse body of work. 
This paper strives to provide a modern presentation of these ``categories for relations'', including their enriched version, further showing how they 
arise as Kleisli categories of symmetric monoidal monads. The resulting taxonomy aims at bringing clarity and 
organisation to the many related concepts and frameworks occurring in the literature.
\end{abstract}    

\section{Introduction}

Category theory, from its very beginnings, was conceived as an abstraction of the notions of set and function. This intuition is clearly expressed in the first sentence of the introduction 
of the well-known book by Mac Lane~\cite{MacLane-1998}

\begin{quote}
Category theory starts with the observation that many properties of 
mathematical systems can be unified and simplified by a presentation 
with diagrams of arrows. Each arrow $f : X \to Y$ represents a function; 
that is, a set $X$, a set $Y$, and a rule $x \mapsto f(x)$ which assigns 
to each element $x \in X$ an element $f(x) \in Y$.
\end{quote}

\noindent
Category theory largely focussed for years on the paradigm that the maps are the counterpart of total functions in the category at hand, and an 
essential use of this metaphor was made in forming the definitions. 
 
The late 1980s witnessed the emergence of a new perspective, aimed at exploring the relational aspects of algebra and logic.
During this period, three fundamental works appeared, shaping the development of categories aimed at abstracting the properties of relations: the papers on \emph{cartesian bicategories} by Carboni and Walters~\cite{Carboni_87} and on \emph{p-categories} by Robinson and Rosolini~\cite{Robinson88},  and the book on \emph{allegories} by Freyd and Scedrov~\cite{freyd1990categories}.

The relevance of such categories has increased in the last years, 
leading several authors to develop this kind of structures.
Such a process has brought to light new categories,
which are very similar in nature but originated in different contexts 
and have different notations.

Relevant examples include \emph{restriction categories}~\cite{Cockett02}, 
which provide an abstract setting for partiality, and \emph{copy/discard categories} (shortly, CD categories)~\cite{cho_jacobs_2019} and its \emph{affine} variant, which is the basis for a recent approach to probability, 
where they are dubbed  \emph{Markov categories}~\cite{Fritz_2020} based on the interpretation of arrows as generalised Markov kernels.

Independently, the notion of \emph{garbage/share monoidal categories} (shortly, gs-monoidal categories) was introduced to present 
a characterisation of \textit{term graphs}~\cite{gadducci1996}, as well as its 2-categorical counterpart suitable to describe term graph \textit{rewriting}~\cite{CorradiniGadducci97}.
Their study was pursued in a series of papers (see e.g. \cite{CorradiniGadducci99,CorradiniGadducci99b} among others), including their application to the functorial semantics of relational and partial algebras~\cite{CorradiniGadducci02,FritzGCT23}.

Another fundamental concept in category theory, which has deep connections with those mentioned above, is the notion of a monad. In particular, the Kleisli category associated with a monad provides a framework for understanding generalised morphisms between objects, capturing both the structure of mappings and the effects described by the monad. In computer science, Kleisli categories have been widely used to model \textit{computations} and \textit{effects}. This connection was first formalised in \cite{Moggi91}, where it was shown how monads provide a rigorous mathematical foundation for these notions.

The purpose of our work is twofold. The first is to arrange and revisit these categorical structures, as well as their 2-categorical versions in the form of preorder-enriched categories, appropriately comparing them to provide a single reference where they can be analysed.
In fact, the large number of similar notions, presented under different names and in different contexts, makes increasingly challenging to navigate the literature and to have a clear framework connecting all these kinds of categories.
The second purpose is to conduct a study on the Kleisli categories of suitable monads on such structures, with the aim of further investigating
 the well-known fact that the Kleisli category of a symmetric monoidal monad on a cartesian category is symmetric monoidal, see \cite{kock71bis}. 
This analysis is extended to the enriched context and this part is, to the best of our knowledge, original to our work.

As a result, we provide an overview of the main categories that abstract the properties of relations, and we show how they are related to each other. To present the various notions, we will use the language of string diagrams~\cite{Selinger2011}, which is widely adopted today in both mathematics and computer science. We believe the exposition can be helpful both to readers with experience in the field, who may appreciate having the relevant material gathered in one place, and to newcomers, for whom we aim to offer a reasonably self-contained introduction.

The key notions on which our taxonomy is built are those of \textit{categories with garbage} and of \textit{share categories}, as well as their dual and enriched versions, which we identify as the core of all the other notions. Our presentation aims to be as modular as possible, to highlight the key differences between the various notions, so that the reader can easily navigate the paper and focus on the specific aspects of interest.
This approach will bring us to alternative characterisations of some known categories, such as restriction and Markov categories. 

Finally, in this work we provide a taxonomy for Kleisli categories in the context of gs-monoidal categories, and we present new examples. Following the philosophy and motivations behind this work, the aim is to showcase, recall, and present the notions of affine and relevant monads  \cite{Kock71,Jacobs1994} (and their enriched version), as well as the characterisations of their Kleisli categories, in the simplest, most general, and modular way possible. 

Examples of these monads arise in a natural way, e.g. with action monads. 
In this context, the affine or relevant monoidal structure of the monad is determined by the condition that the base category is connected or special, respectively. 
Similar examples are obtained by taking instances of the semiring monad, providing also case studies for the enriched case.

It is worth observing that there are other approaches to presenting taxonomies for categories of relations, which do not take suitable monoidal categories as the central organising principle, adopting for instance a perspective based on double categories (see e.g.\ \cite{HoshinoN25}).
Moreover, several methodologies for analysing and decomposing algebraic structures are available in the literature. Distributive laws and factorisation systems, in particular, provide well-established techniques for decomposing monoids, bialgebras, and related structures, and they play a central role in the standard theory of composing monads (see \cite{rosebrugh2002distributive,lack2004composing,cheng2011iterated}). These tools provide a modular perspective that complements the gs-monoidal and monadic taxonomies developed in this work. While we briefly discuss the epi–mono factorisation in the context of regular categories, a systematic comparison with these decomposition techniques falls outside the scope of this paper. A strength of the gs-monoidal approach, when compared with the alternatives mentioned above, lies in the simplicity of the underlying notions and in the axiomatic style of the presentation: richer structures are obtained by adding small and transparent sets of axioms. In addition, the use of string diagrams -- which  gained considerable prominence in contemporary categorical practice -- provides an intuitive yet rigorous language that facilitates the formulation and the comparison of the various relational notions.

The paper is structured as follows. In Section~\ref{sec:on gs} we provide a background on the notions of interest, notably gs-monoidal categories, 
focussing on their relationships with Markov and restriction categories.  In Section~\ref{sec:taxonomy kleisli} we characterise the structure of Kleisli categories
for symmetric monoidal monads.
In Section~\ref{sec:oplax cartesian categories} we study oplax cartesian categories, the order-enriched version of gs-monoidal categories, and their relationship
with cartesian bicategories. Finally, in Section~\ref{sec:conclusions} we draw some conclusions and outline future research directions.

The paper presents in details only the proof of the results which we believe to be original, such as those concerning the enriched structures, while providing references for the others.
An appendix recalls some standard categorical notions and collects some additional definitions.

\section{A taxonomy of gs-monoidal categories}

\label{sec:on gs}
The overall, one-dimensional taxonomy of the various ``categories for relations'' we are going to analyse here is represented by the diagram below, where the presence of an arrow 
$X \to Y$ means that a category falling in the class $X$ (e.g. Markov categories) also belongs to the class $Y$ (categories that are 
either gs-monoidal or with projections)

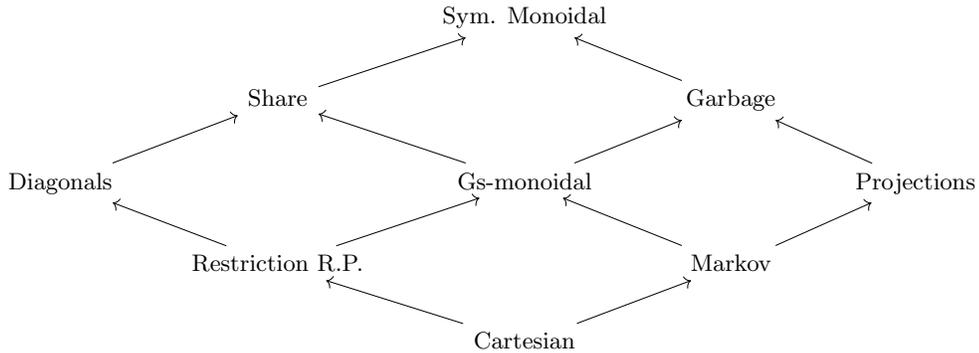
\begin{center}

\small
\begin{tikzcd}
	&& {\text{Symmetric Monoidal}} \\
	& {\text{Garbage}} && {\text{Share}} \\
	{\text{Projections}} && {\text{GS-monoidal}} && {\hspace{-.5cm}\text{Diagonals}} \\
	& {\text{Markov}} && {\hspace{-.5cm}\text{Cartesian restriction}} \\
	&& {\text{Cartesian monoidal}}
	\arrow[from=2-2, to=1-3]
	\arrow[from=2-4, to=1-3]
	\arrow[from=3-1, to=2-2]
	\arrow[from=3-3, to=2-2]
	\arrow[from=3-3, to=2-4]
	\arrow[from=3-5, to=2-4]
	\arrow[from=4-2, to=3-1]
	\arrow[from=4-2, to=3-3]
	\arrow[from=4-4, to=3-3]
	\arrow[from=4-4, to=3-5]
	\arrow[from=5-3, to=4-2]
	\arrow[from=5-3, to=4-4]
\end{tikzcd}
\end{center}

\bigskip
Monoidal categories were introduced by B\'enabou \cite{benabou_1963} and later finitely axiomatised by Mac Lane in \cite{maclane_1963}, with the term ``monoidal category'' first appearing in a 1966 paper by Eilenberg and Kelly \cite{closed_categories_1966}. Nowadays, the literature on monoidal categories is large. Hence, for a standard presentation we refer to \cite{MacLane-1998,HCA2}.

A monoidal category is a category equipped with a tensor product operation and a unit object, satisfying certain coherence conditions. It provides a framework for studying structures where objects can be combined and interactions are modelled algebraically.
Cartesian categories are one of the leading examples of monoidal categories, where the tensor product is the categorical product.

As shown by Fox in \cite{Fox:CACC},  cartesian categories (with a choice of products) are precisely symmetric monoidal categories in which every object is equipped with a commutative 
comonoid structure given by $\nabla_X: X\to X\times X$ and  $!_X:X\to I$, both natural in $X$. 
We will refer to these structures as  \textit{cartesian monoidal categories}.

By relaxing the requirement about the existence and the naturality of these two families of arrows, one obtains a series of well-known categories.

Requiring the existence of both families but not their naturality, one obtains the notion of gs-monoidal category, introduced by Corradini and Gadducci in \cite{gadducci1996,CorradiniGadducci97} in the context of algebraic presentations of graphical formalisms. These categories abstract the properties of cartesian product of sets in the category of relations on one hand, and on the other they have the right structure to distinguish between relations, partial functions, total relations and functions.

Requiring the existence of both families but  naturality only for comultiplication $\nabla_X$, one obtains categories apt to abstract the notion of \textit{partial} function. These categories have been presented in equivalent forms with the names \textit{ restriction categories with restriction products} by Cockett and Lack in \cite{Cockett02,Cockett03,Cockett07}, \textit{p-categories} by Rosolini and Robinson in \cite{Robinson88} and \textit{partial categories} by Curien and Obtulowicz in \cite{CURIEN198950}.
Currently, the most accepted name seems \emph{cartesian restriction categories}~\cite{Cockett07}.
%
Requiring only the existence of comultiplication $\nabla_X$ and its naturality, one obtains the notion of \textit{categories with diagonals}, introduced by Jacobs in \cite{Jacobs1994} in the context of linear logic.
The existence of comultiplication $\nabla_X$ without naturality gives rise to the notion of \textit{share categories}.

Requiring the existence of both families but naturality only for counit $!_X$, one obtains categories apt to abstract the notion of \textit{total} relation. These categories have been presented in equivalent forms with the names \textit{affine CD categories} by Cho and Jacobs in \cite{cho_jacobs_2019} and Markov categories by Fritz in \cite{Fritz_2020} in the context of categorical probability theory.
Requiring only the existence of discharger $!_X$ and its naturality, one obtains the notion of \textit{categories with projections}, introduced by Jacobs in \cite{Jacobs1994} in the context of linear logic.
The existence of counit $!_X$ without naturality gives rise to the notion of \textit{categories with garbage}.

The categories we recalled above are among the most commonly discussed in the literature on ``categories for relations''. 
In fact, the vast majority of the related papers assumes at least the occurrence of a monoidal category
where each object is a comonoid, possibly equipped with further structure, i.e. requiring the presence of additional families of arrows.
The survey will thus focus on enucleating the connections within the components of our taxonomy,
viewed as the building blocks of the whole spectrum of ``categories for relations'', further exploring 
their dual and preorder-enriched variants.

\subsection{Share categories}
\label{share}
 Here and in the following, we will use string diagrams to represent morphisms in monoidal categories. In this formalism, everything is considered up to associativity of the monoidal product and cancellation of the unit. No generality is lost, since one can either appeal to the strictification of the categories in question \cite[Proposition 3.28]{fongspivak2019}, or incorporate the required coherence isomorphisms into the axioms. This makes it possible to easily introduce the usual auxiliary structures of symmetric monoidal categories (see also Remark 2.3 in \cite{FritzGCT23}).

    \begin{mydefinition}
        A \textbf{share category} is a symmetric monoidal category $(\mC, \otimes, I)$ together with
        a commutative cosemigroup structure for each object $X$,
        consisting of a comultiplication 
        \ctikzfig{copy}
        which is coassociative and cocommutative
        \ctikzfig{comonoid_share_cat}
        These cosemigroup structures must be multiplicative with respect to the monoidal structure, meaning that
        they satisfy the equations
\ctikzfig{comon-struct-mult-share-cat}
\end{mydefinition}

We may refer to $\nabla_X \colon X \to X \otimes X$ as the \textbf{duplicator}.
  \begin{myexample}
      The category $(\Rel, \times,\{\bullet\})$ of sets and relations with the composition of $a\subseteq X\times Y$ with $b\subseteq Y\times Z$  given by
	\[b\circ a:=\{(x,z) \mid \exists y\in Y,(x,y)\in a \wedge (y,z)\in b \}\subseteq X\times Z\]
    and the monoidal operation given by the direct product of sets is a share category. In this case,  $\nabla_X \colon X \to X \times X$ is given by the \emph{function} $x\mapsto \langle x, x \rangle$.
\end{myexample}
\begin{myexample}\label{ex_Nat_is_share_cat}
	The monoidal category $(\mathcal{N},+,0)$, where $\mathcal{N}$ is the posetal category of natural numbers and the monoidal operation $+$ is given by the usual sum of natural numbers,
	 is a share category, where the comultiplication is given by the arrow $n\leq n+n$.
\end{myexample}
\begin{myexample}\label{pfunasshare}
	The leading example of share category is the category of sets and \emph{partial functions}. The monoidal operator is given by the direct product of sets and the monoidal 
	by the singleton set $\{\bullet\}$. The comultiplications are 
given by the functions $X \to X \times X$ such that $a \mapsto \langle a, a \rangle$. We will see in Example~\ref{example: due relazioni} how the share structure of this category can be derived from the gs-monoidal structure of the category of sets and relations.
\end{myexample}
\begin{myexample}\label{ex_Set}
Note that also the monoidal category $(\mathbf{Set},\times,\{\bullet\})$ of sets and functions is a share category, again  with the direct product as the monoidal operator and
as comultiplications the functions $X \to X \times X$ such that $a \mapsto \langle a, a \rangle$. \end{myexample}

The following lemma shows that a share structure can be equivalently given in terms of a monoidal transformation (see Appendix \ref{sec:lax_app}).

\begin{mylemma}\label{rmk:equiv share}
	Given a symmetric monoidal category $(\mC, \otimes,I)$ and the two trivial strong symmetric monoidal functors given by the identity functor
\[(\id, \id_I, \id_{+\otimes -}):\mC \to \mC \]
and the functor $\otimes(-,-):\mC\to\mC$ which sends $X$ to $X\otimes X$
\[(\otimes(-,-), \lambda_I, s_{+,-}):\mC\to \mC\]
where $s_{X,Y}\colon (X\otimes Y)\otimes(X\otimes Y)\to (X\otimes X)\otimes(Y\otimes Y)$ is the symmetry defined as 
$\alpha^-_{X,X,Y\otimes Y} \circ (\id_X\otimes\, \alpha_{X,Y,Y}) \circ (\id_X \otimes (\gamma_{Y,X}\otimes\id_Y)) \circ (\id_X \otimes\, \alpha^-_{Y,X,Y}) \circ \alpha_{X,Y,X\otimes Y}$, where $\alpha$ denotes the associator and $\gamma$ the braiding of the monoidal structure.
A \textbf{share} structure is given by a (not necessarily natural) monoidal transformation
\[\nabla_{-}:(\id,\id_I,\id_{+\otimes -})\to (\otimes(-,-), \lambda_I, \id_+\otimes\, \gamma_{+,-}\otimes \id_{-})\]
satisfying the additional conditions 
\[(\id_X \otimes \nabla_X) \circ \nabla_X = \alpha_{X,X,X} \circ (\nabla_X \otimes \id_X) \circ \nabla_X\]
\[\gamma \circ \nabla_X=\nabla_X.\]
\end{mylemma}
\begin{proof}

Commutativity of Diagram (\ref{eq:monoidal transformation}) in the Definition~\ref{laxtras}
of monoidal transformation $\nabla_{-}$ provides the equations for $\nabla_{X\otimes Y}$ and $\nabla_I$. The additional conditions correspond to coassociativity and cocommutativity of $\nabla$,
 respectively.
\end{proof}

\begin{myremark}
It is now immediate to see that a share category is a \emph{category with diagonals}~\cite[Def. 2.1]{Jacobs1994}
if $\nabla_{-}$ is a natural transformation.
\end{myremark}

\begin{mydefinition}\label{def relevant functor}
	For share categories $\mC$ and $\mD$, a functor $\freccia{\mC}{F}{\mD}$ equipped with a lax symmetric monoidal structure
   \[
				\freccia{\otimes \circ \, (F\times F)}{\psi}{F\circ \otimes}, \qquad \freccia{I}{\psi_0}{F(I)} 
			\]
is \textbf{relevant} if the following diagram commutes for all $X$ in $\mC$

\begin{equation}\label{diagram: lax relevant}
\begin{tikzcd}[column sep=tiny]
	{F(X)} && {F(X\otimes X)} \\
	& {F(X)\otimes F(X)}
	\arrow["{F(\nabla_X)}", from=1-1, to=1-3]
	\arrow["{\nabla_{FX}}"', from=1-1, to=2-2]
	\arrow["{\psi_{X,X}}"', from=2-2, to=1-3]
\end{tikzcd}
\end{equation}
\end{mydefinition}

\begin{mydefinition}
    An arrow $f:X\to Y$ in a share category is called \textbf{copyable} or \textbf{functional} if
    \ctikzfig{functional}

\end{mydefinition}
\begin{myexample}
	In the category $\Rel$, the copyable arrows are precisely the partial functions.
\end{myexample}
Therefore, the notion of category with diagonals~\cite[Def. 2.1]{Jacobs1994} can be easily rephrased in terms of share category.
\begin{mylemma}
A share category has diagonals if and only if
every arrow is functional.
\end{mylemma}
\begin{myexample}
	The share categories $(\mathcal{N},+,0)$  and $(\mathbf{Set},\times,\{\bullet\})$ 
	presented in Example~\ref{ex_Nat_is_share_cat} and Example~\ref{ex_Set} respectively are both categories with diagonals.
\end{myexample}
\begin{mycorollary}
For a share category $\mC$, the sub-category $\mathbf{Fun}(\mC)$ consisting of functional arrows is a category with diagonals.
\end{mycorollary}
\begin{myremark}
	Observe that $\mathbf{Fun}$ provides a functor from the category of share categories and strict relevant functors to the full subcategory of categories with diagonals and strict relevant functors.
\end{myremark}

We close recalling the notion of positivity from~\cite[Def.~11.22]{Fritz_2020}.

\begin{mydefinition}\label{positive}
    A share category is called \textbf{positive} 
    if for every pair of arrows $f:X\to Y$ and $g: Y \to W$ such that 
    $g \circ f$ is functional then
\ctikzfig{positivity}
\end{mydefinition}

\begin{myexample}
	Every category with diagonals is positive. The category of sets and partial functions of Example~\ref{pfunasshare} is a positive share category.
\end{myexample}
\subsection{Categories with garbage}
\label{garbage}

\begin{mydefinition}
    A \textbf{category with garbage} is a symmetric monoidal category $(\mC, \otimes, I)$ together with
     a distinguished arrow for each object $X$
\ctikzfig{del}
        These arrows must be multiplicative with respect to the monoidal structure, meaning that
        they satisfy the equations
\ctikzfig{axiom_grb_cat}
\end{mydefinition}
We may refer to $!_X : X \to I$ as the \textbf{discharger}.
\begin{myexample}
	  The category $(\Rel, \times,\{\bullet\})$ of sets and relations is a category with garbage. In this case,  $!_X \colon X \to \{\bullet\}$ is given by the \emph{relation} $\{(x,\bullet) \mid x \in X\}$.	
\end{myexample}
\begin{myexample}\label{ex_Nat_op_is_garbage_cat}
	The monoidal category $(\mathcal{N}^{\mathrm{op}},+,0)$, where $(\mathcal{N},+,0)$ is the category defined in Example~\ref{ex_Nat_is_share_cat}, is a category with garbage.
\end{myexample}

\begin{myexample}
	The leading example of a category with garbage is the category of sets and \emph{total relations}, i.e.\ relations $R \subseteq X \times Y$ such that for all $x \in X$ there exists $y \in Y$ with $(x,y) \in R$. See Example~\ref{example: due relazioni} for the monoidal structure of this category.  While a leading counterexample is the share category of partial functions of Example~\ref{pfunasshare}
	(and of course vice versa).	
	Also for this case we will see in Example~\ref{example: due relazioni} how the garbage structure of this category can be derived from the gs-monoidal structure of the category of sets and relations.
\end{myexample}

\begin{myexample}\label{ex_Set_G}
The monoidal category $(\mathbf{Set},\times,\{\bullet\})$ of sets and functions with the direct product of Example \ref{ex_Set} as the monoidal operator 
is also a category with garbage, where the arrows are given by the function $a \mapsto \bullet$.
\end{myexample}

The following lemma shows that garbage structure can be equivalently given in terms of a monoidal transformation.

\begin{mylemma}\label{rmk:equiv garbage}
	Given a symmetric monoidal category $(\mC, \otimes,I)$ and the two trivial strong symmetric monoidal functors given by the identity
\[(\id, \id_I, \id_{+\otimes -}):\mC \to \mC \]
and the constant value $I$ functor
\[(I, \id_I, \lambda_I^{-1}):\mC\to \mC\]
a \textbf{garbage} structure is given by a (not necessarily natural) monoidal transformation
\[!_{-}:(\id,\id_I,\id_{+\otimes -})\to (I, \id_I, \lambda_I^{-1}).\]
\end{mylemma}

\begin{myremark}
Similarly to what occurs for share categories, a category with garbage is a \emph{category with projections}~\cite[Def. 2.1]{Jacobs1994}
if $!_{-}$ is a natural transformation.
Notice that categories with projections are also called \emph{semicartesian categories} in \cite{Gerhold2022,Fritz_2020}.
\end{myremark}

\begin{mydefinition}\label{def garbage functor}
	For categories with garbage $\mC$ and $\mD$, a functor $\freccia{\mC}{F}{\mD}$ equipped with a lax symmetric monoidal structure
   \[
				\freccia{\otimes \circ \, (F\times F)}{\psi}{F\circ \otimes}, \qquad \freccia{I}{\psi_0}{F(I)} 
			\]
is \textbf{affine}  if the following diagram commutes for all $X$ in $\mC$
\begin{equation}\label{diagram: lax affine}
\begin{tikzcd}[column sep=tiny]
	F(X) && {F(I)} \\
& I
\arrow["{F(!_X)}", from=1-1, to=1-3]
\arrow["{!_{F(X)}}"', from=1-1, to=2-2]
\arrow["{\psi_0}"', from=2-2, to=1-3]
\end{tikzcd}
\end{equation}
\end{mydefinition}

\begin{mydefinition}
	An arrow $f:X\to Y$ in a category with garbage is called \textbf{discardable} or \textbf{total} if 
	 \ctikzfig{full}

 \end{mydefinition}
\begin{myexample}
	In the category $\Rel$, the discardable arrows are precisely the total relations.
\end{myexample}
\begin{mylemma}
A category with garbage has projections if and only if every arrow is total.
\end{mylemma} 

\begin{myremark}\label{terminal}
Notice that in a garbage category every arrow is total if and only if the object $I$ is terminal.
\end{myremark} 

\begin{mycorollary}
For a garbage category $\mC$, the sub-category $\mathbf{Tot}(\mC)$ of total arrows is a category with projections.
\end{mycorollary}
\begin{myremark}
	Observe that $\mathbf{Tot}$ provides a functor from the category of garbage categories and strict affine functors to the full sub-category of categories with projections and strict affine functors.
\end{myremark}

\subsection{GS-monoidal categories}
\label{gs-mon}
The original notion of \emph{gs-monoidal category} introduced in~\cite{gadducci1996,CorradiniGadducci97} can be presented combining the notions of share and garbage category.
\begin{mydefinition}\label{dfn: gs-monoidal}
	 A \textbf{gs-monoidal category} is a symmetric monoidal category $(\mC, \otimes, I)$ with share and garbage structures such that for each object $X$
	 \ctikzfig{copy_del_equation}
	\end{mydefinition}

In other words, each object is equipped with a commutative comonoid structure.

We can now close this section with a result that is basically stated in~\cite{Fox:CACC}, but see e.g. \cite[\S 6.4]{mellies2009categorical} for a contemporary presentation.

\begin{mylemma}
A gs-monoidal category is cartesian monoidal if and only if $\nabla_{-}$ and $!_{-}$ are natural transformations.
\end{mylemma}

\begin{mycorollary}\label{cor_TFun_is_cartesian}
The sub-category $\mathbf{TFun}(\mC)$ of total and functional arrows is cartesian monoidal.
\end{mycorollary}

   \begin{myexample}\label{example: due relazioni}
   The category $(\mathbf{Set}, \times, \{\bullet\})$ of sets and functions with the direct product is gs-monoidal, and in fact cartesian monoidal.
    The category $(\Rel, \times,\{\bullet\})$ of sets and relations with the composition of $a\subseteq X\times Y$ with $b\subseteq Y\times Z$  given by
	\[b\circ a:=\{(x,z) \mid \exists y\in Y,(x,y)\in a \wedge (y,z)\in b \}\subseteq X\times Z\]
    and the monoidal operation  
    given by the direct product of sets 
    is the leading example of gs-monoidal category \cite{CorradiniGadducci02}. 
    In this category, the copyable arrows are precisely the partial functions, and the discardable arrows are the total relations.
\end{myexample}
	
\begin{myexample}
Recently, an alternative category of relations $(\Rel^{\forall}, \otimes,\{\bullet\})$ has been investigated~\cite{bonchi2024,bonchi_trotta_2024}. The category $\Rel^{\forall}$ has the same objects and arrows as $\Rel$, but the composition of a relation $a\subseteq X\times Y$ with $b\subseteq Y\times Z$ is given by
	\[b\circ a=\{(x,z) \mid \forall y\in Y,(x,y)\in a \vee (y,z)\in b \}\subseteq X\times Z\]
	and the identity arrow is given by $\id_X=\{\left(x,y\right)| x\neq y\}\subseteq X\times X$.
The tensor product is defined on objects as the direct product of sets (as in the previous case), and on two arrows $a\subseteq X\times Y$ and $c\subseteq Z\times V$
is given by 
\[ a\otimes c= \{\left(\left(x,z\right), \left(y,v\right)\right) |\ \left(x,y\right)\in a \vee \left(z,v\right) \in c   \}\]
	This category is gs-monoidal with the following structure arrows
	\[\nabla_X=\{\left(x,\left( y,z\right) \right)| x\neq y \vee x\neq z\}\subseteq X\times (X \times X)\qquad !_X= \emptyset\subseteq X\times I\]
	Note that $\Rel^{\forall}$ is isomorphic to $\Rel$ via a ``complement functor'', which is strict symmetric monoidal~\cite{bonchi2024}.
	Although we are not aware of a characterisation of copyable relations for this case, it is easy to 
	show that the discardable relations are precisely those $a\subseteq X\times Y$ such that
	$\forall x \in X.\; \{y \in Y \mid (x, y) \in a\} \neq Y$.
	
    \end{myexample}
     
We close now with an observation from \cite[Prop. 2.10\&2.11]{FritzGCT23}.

\begin{mylemma}\label{prop: 2.10 e 2.11 FGTC23}
	Let $(\mC,\otimes, I)$ be a gs-monoidal category. Then $\mathbf{Fun}(\mC)$, $\mathbf{Tot}(\mC)$, and $\mathbf{TFun}(\mC)$ are gs-monoidal sub-categories of $\mC$.
\end{mylemma}

    As for functors between symmetric monoidal categories, also functors between gs-monoidal categories come in several variants. 
   
\begin{mydefinition}\label{def gs monoidal functor}
	For gs-monoidal categories $\mC$ and $\mD$, a functor $\freccia{\mC}{F}{\mD}$ equipped with a lax symmetric monoidal structure
   \[
				\freccia{\otimes \circ \, (F\times F)}{\psi}{F\circ \otimes}, \qquad \freccia{I}{\psi_0}{F(I)} 
			\]
is \textbf{gs-monoidal} if it is both relevant and affine.
\end{mydefinition}

\subsection{Recurring examples: spans and (weighted) relations}\label{ex_spans_are_gs}\label{example: span + gs monoidale}

 Recall from \cite{Bruni2003} that the category $\mathbf{Span}(\mA)$ of \emph{spans} associated with a category with finite limits  $\mA$ is a gs-monoidal category. The category $\mathbf{Span}(\mA)$ has the same objects 
	as $\mA$, and an arrow from $X$ to $Y$ is a \emph{span}, i.e. an equivalence class of diagrams of the form
		 $(X  \xleftarrow{a_X} A \xrightarrow{a_Y} Y)$ of $\mA$, where $(A,a_X,a_Y)\sim (B,b_X,b_Y)$ if there exists an isomorphism $i:A\to B$ such that $b_X\circ i=a_X$ and $b_Y\circ i=a_Y$. Identities and composition are defined as follows
	\begin{itemize}
		\item the identity of $X$ is the span $X \xleftarrow{\id_X}X  \xrightarrow{\id_X} X$;
		\item the composition of spans $X\leftarrow  A \xrightarrow{f} Y$ and $Y\xleftarrow{g}  B \rightarrow Z$ is given by the span $X\leftarrow A\times_{Y}B \to Z$ obtained through the pullback of $f$ and $g$
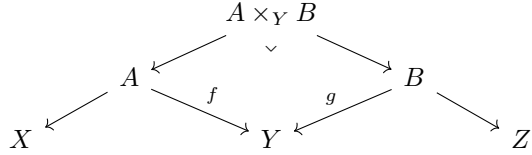
		\[
		\begin{tikzcd}[row sep=small]
			&& {A\times_YB} \\
			& A && B \\
			X && Y && Z
			\arrow[from=1-3, to=2-2]
			\arrow[from=1-3, to=2-4]
			\arrow["\lrcorner"{anchor=center, pos=0.125, rotate=-45}, draw=none, from=1-3, to=3-3]
			\arrow[from=2-2, to=3-1]
			\arrow["f", from=2-2, to=3-3]
			\arrow["g"', from=2-4, to=3-3]
			\arrow[from=2-4, to=3-5]
		\end{tikzcd}\]
	\end{itemize}
	The tensor product is given via the categorical product $\times$ of $\mathcal{A}$ and the gs-monoidal structure is given by the arrows
	\[
	\nabla_X = (X \xleftarrow{\id} X \xrightarrow{\nabla_X} X \times X), \qquad !_X = (X \xleftarrow{\id} X \xrightarrow{!_X} 1).
\]
Although the duplicator  $\nabla_X$ and the discharger $!_X$ are natural in $\mathcal{A}$, in general they are not in $\mathbf{Span}(\mA)$ as observed in \cite{Bruni2003} and, in an equivalent way, in \cite[Ex. 2.1.4(10)]{Cockett02} where it is noted that $\mathbf{Span}(\mA)$ is not a restriction category (see also Proposition \ref{gs as restriction}). We will refer to this gs-monoidal category as $(\mathbf{Span}(\mathcal{A}),\times,1)$.

\begin{myremark}
It is worth observing that the sub-category $\mathbf{Span}_m(\mA)$ of  $\mathbf{Span}(\mA)$  of spans whose left leg is a mono is gs-monoidal as well (since monos are stable under pullbacks) and in fact it has diagonals. 
This category is a particular instance of a \emph{category of partial maps} $\mathbf{Par}(\mA, \mathcal{M})$ associated with a stable system of monics $\mathcal{M}$ presented in \cite[Sec.~3.1]{Cockett02}.

Moreover, the sub-category $\mathbf{Span}_e(\mA)$ of $\mathbf{Span}(\mathcal{A})$ 
of spans whose left leg is a split epimorphism\footnote{An arrow $f:A\to B$ splits if it has a \textit{section}, i.e.\ an arrow $s:B\to A$ such that $f\circ s=\id_B$.} 
has projections, see \cite[Prop.~5.4]{FritzGCT23}.

\end{myremark}

	If $\mathcal{A}$ is an extensive category  (see  Definition~\ref{def:extensive category} and \cite{carboni1993introduction} for details) with finite limits, then the category $\mathbf{Span}(\mA)$
	has another gs-monoidal structure. The tensor product is given by the categorical sum $+$ and the gs-monoidal structure is given by the arrows
	\[
	\nabla_X = (X \xleftarrow{(\id, \id)} X+X \xrightarrow{\id} X + X), \qquad !_X = (X \xleftarrow{\cobang} 0 \xrightarrow{\id} 0).
\]
This category is actually cartesian and we will refer to it as $(\mathbf{Span}(\mA),+,0)$.

\subsubsection{From spans to relations}
\label{example: relazioni x sono gs monoidale }
	It is well-known that the notion of category of relations $\Rel$ of Example \ref{example: due relazioni} can be generalised to regular categories (see \cite[Ex. 1.4]{CARBONI198711}) or, alternatively, to categories equipped with a proper, stable factorization system (see \cite{Kelly_92}). Indeed, let $\mathcal{A}$ be a regular category (Definition~\ref{def:regular category}), then consider the category of relations $\mathbf{Rel}(\mathcal{A})$
	whose objects are those of $\mathcal{A}$ and arrows are given by jointly monic spans. The composition of relations $X\xleftarrow{f_X}  A \xrightarrow{f_Y} Y$ and $Y\xleftarrow{g_Y}  B \xrightarrow{g_Z} Z$ is given by first considering the span $X\leftarrow A\times_{Y}B \to Z$ obtained by taking the pullback of $f$ and $g$,
	and then considering the regular-epi/mono factorization of the induced arrow $\langle f_X\circ g_Y',g_Z\circ f_Y'\rangle:A\times_{Y}B\to X\times Z$
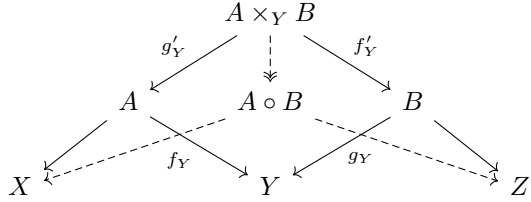
	\[
	\begin{tikzcd}
		&& {A\times_{Y}B} \\
		& A & {A\circ B} & B \\
		X && Y && Z
		\arrow["{g_Y'}"', from=1-3, to=2-2]
		\arrow[dashed, two heads, from=1-3, to=2-3]
		\arrow["{f_Y'}", from=1-3, to=2-4]
		\arrow[from=2-2, to=3-1]
		\arrow["f_Y"', from=2-2, to=3-3]
		\arrow[dashed, from=2-3, to=3-1]
		\arrow[dashed, from=2-3, to=3-5]
		\arrow["g_Y", from=2-4, to=3-3]
		\arrow[from=2-4, to=3-5]
	\end{tikzcd}\]
	 The categorical product $\times$ induces a gs-monoidal structure as in $( \mathbf{Span}(\mathcal{A}),\times,1)$, and this gs-monoidal category will be denoted as $(\mathbf{Rel}(\mathcal{A}), \times,1)$.
	From a logical perspective, the main reason why the $\mathbf{Rel}$-construction can be generalised in this setting is that regular categories are able to properly deal with the $(\exists,=,\wedge,\top)$-fragment of first-order logic (the \emph{regular fragment}). In particular, they have an ``internal'' notion of existential quantifier, which allows to mimic the usual composition of relations (which is defined via the existential quantifier).
	Similarly, if $\mathcal{A}$ is also extensive, then the coproduct induces a cartesian monoidal structure as in $(\mathbf{Span}(\mathcal{A}),+,0)$, which
	we will denote as $(\mathbf{Rel}(\mathcal{A}),+,0)$.

\begin{myremark}\label{spas as rel}
For a regular category $\mathcal{A}$ in which every regular epi splits, the category of relations  is equivalent to a category of spans.
Namely, $\mathbf{Rel}(\mathcal{A})$ is equivalent to the category whose objects are those of $\mathcal{A}$ and whose arrows are equivalence classes of diagrams of the form
$(X  \xleftarrow{a_X} A \xrightarrow{a_Y} Y)$ of $\mA$, where $(A,a_X,a_Y)\sim (B,b_X,b_Y)$ if there exist two arrows $h:A\to B$ and $k:B\to A$ such that 
$$\begin{cases}
	b_X\circ h=a_X \\
	b_Y\circ h=a_Y
\end{cases} \qquad \begin{cases}
	a_X\circ k=b_X \\
	a_Y\circ k=b_Y
\end{cases} $$
Indeed, every span $(X  \xleftarrow{a_X} A \xrightarrow{a_Y} Y)$ is in the same equivalence class of $(X  \xleftarrow{m_X} M \xrightarrow{m_Y} Y)$, where $m:=\angbr{m_X}{m_Y}$ is the image of the factorization of $\angbr{a_X}{a_Y}$
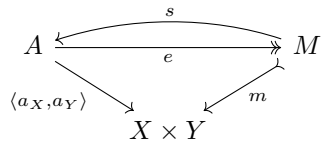
\[
\begin{tikzcd}
	A && M \\
	& {X\times Y}
	\arrow["e"', two heads, from=1-1, to=1-3]
	\arrow["{\langle a_X, a_Y\rangle}"',from=1-1, to=2-2]
	\arrow["s"', curve={height=12pt}, from=1-3, to=1-1]
	\arrow["m", tail, from=1-3, to=2-2]
\end{tikzcd}\]
In particular, the above observation holds for $\mathcal{A}=\mathbf{Set}$, where this kind of spans are referred to 
as \textit{garbage equivalent} ones \cite{Gadducci08}.
\end{myremark}

\subsubsection{And now, weighted relations}
\label{WR}\label{example: monad semiring}
While the previous section generalised relations by changing the base category, an alternative approach
is to equip each relation with a \emph{weight}.
Consider a semiring $(M,\oplus,\odot, 0, 1)$. It is easy to check that 
there exists a functor $\mathcal{M}:\mathbf{Set}\to\mathbf{Set}$ that sends every set $X$ to 
 \[\mathcal{M}(X)=\left\{ h:X \to M \ |\  h\  \text{has finite support}\right\}\]
where \emph{finite support} means that $h(x)\neq 0$  for a finite number of elements $x\in X$, and every function $f:X\to Y$ to the function $\tilde{f}:\mathcal{M}(X)\to \mathcal{M}(Y)$ which sends every $M$-valued function $h:X\to M$ with finite support 
to 
\[\tilde{f}(h)(y)= \underset{x\in f^{-1}(y)}{\bigoplus} h(x)\]
Explicitly, the image of the composition $ g \circ f$ of $f:X\to Y$ and $g:Y\to Z$ is
\[\widetilde{g \circ f}(h)(z)= \underset{y\in g^{-1}(z)}{\bigoplus}(\underset{x\in f^{-1}(y)}{\bigoplus} h(x))\]

Recall that $(\mathbf{Set}, \times, \{\bullet\})$ is cartesian monoidal with respect to the direct product.
The above functor is lax symmetric monoidal with respect to that monoidal structure, with the obvious 
coherence arrows 
\[\psi_{X,Y}:\mathcal{M}(X)\times \mathcal{M}(Y)\to \mathcal{M}(X\times Y)
\qquad\psi_0:\{\bullet\}\to \mathcal{M}(\{\bullet\})\]
given by $\psi_{X,Y}(h,k)(x,y)=h(x)\odot k(y)$ and $\psi_0(\bullet)(\bullet)=1$.

There are two additional endofunctors $\mathcal{M}_e$ and $\mathcal{M}_u$ on $\mathbf{Set}$ 
that are lax symmetric monoidal with respect to the monoidal structure $(\mathbf{Set}, \times, \{\bullet\})$. These are defined as follows
\[\mathcal{M}_e(X)=\left\{ h:X \to M \ |\  h\  \text{has support at most one and is idempotent}\right\}\]
\[\mathcal{M}_u(X)=\left\{ h:X \to M \ |\  h\  \text{has finite support and is normalised}\right\}\]
where idempotent means that $\forall x \in X.\, h(x) = h(x) \odot h(x)$ and normalised that $\bigoplus_{x \in X}h(x) = 1$.
Besides being lax symmetric monoidal functors, $\mathcal{M}_e$ is relevant and $\mathcal{M}_u$ is affine.

\begin{myremark}\label{power}
The setting above is general enough to recover various relational structures defined in the literature.
	If e.g. $M$ is the Boolean semiring $\{0,1\}$, then $\mathcal{M}$ is the lax symmetric monoidal 
	functor $\mathcal{P}$ associating to $X$ its finite subsets, and it is neither relevant nor affine.
	The relevant functor $\mathcal{P}_e$ is restricted to subsets of at most one element,
	while the affine functor $\mathcal{P}_u$ is restricted to subsets with at least one element. 

\end{myremark}

\subsection{On duality}
The definitions of the gs-monoidal structure can be easily dualised.

\begin{mydefinition}
    A \textbf{cogs-monoidal category} is a symmetric monoidal category $(\mC, \otimes, I)$
    such that its dual category $\mC^{op}$ is a gs-monoidal category.
\end{mydefinition}
    
In words, each object $X$ is equipped with a monoid structure
$\Delta_X : X \otimes X \to X $ and $\cobang_X: I \to X$, namely there exist two arrows
\ctikzfig{co_copy_co_del}
satisfying the obvious axioms.

\begin{myremark}
We assume, without repeating the statements, that all the remarks in the previous sections concerning sharing and garbage can be dualised, so that e.g. 
$\Delta_X : X \otimes X \to X $ and $\cobang_X: I \to X$ can be described as suitable monoidal transformations, and that their naturality boils down to require 
that a cogs-monoidal category is actually cocartesian monoidal. See also Appendix~\ref{section: appendix dual results}.
\end{myremark}

\begin{myexample}\label{example: span x + are also cogs}
	The categories of spans introduced in Section \ref{ex_spans_are_gs} 
	are cogs-monoidal. Indeed,
	if $\mathcal{A}$ has finite limits, $(\mathbf{Span}(\mathcal{A}),\times,1)$ has a cogs-monoidal structure given by the arrows
	\[\Delta_X  = (X\times X \xleftarrow{\nabla_X} X \xrightarrow{\id} X ), \qquad \cobang_X = (1 \xleftarrow{!} X \xrightarrow{\id} X).
\]
If $\mathcal{A}$ is also an extensive category, $(\mathbf{Span}(\mathcal{A}),+,0)$ has cogs-monoidal structure 
(and it is in fact cocartesian) given by the arrows
\[	\Delta_X = (X+X \xleftarrow{\id} X+X \xrightarrow{(\id,\id)} X ), \qquad \cobang_X = ( 0\xleftarrow{\id} 0 \xrightarrow{\cobang} X).
\]
The same considerations apply for the category $\mathbf{Rel}(\mathcal{A})$ and the gs-monoidal structures presented in Examples \ref{example: relazioni x sono gs monoidale }. 
\end{myexample}

\begin{mydefinition}
    A \textbf{bigs-monoidal category} is a symmetric monoidal category $(\mC, \otimes, I)$ that is 
    both a gs-monoidal and a cogs-monoidal  category.
\end{mydefinition}

Notice that the definition of a bigs-monoidal category includes no requirement regarding the interaction between the monoid and the comonoid structures. The rest of this section is devoted to studying some relevant axioms that establish the possible ways these structures interact, see \cite{heunen2012lectures} for further details.

\begin{mydefinition}\label{dfn: specal and connected bimonoid}
	A bigs-monoidal category $(\mathcal{C},\otimes,I)$ is \textbf{special} if the law below on the left holds, while
	it is \textbf{connected} if the law below on the right holds
	\ctikzfig{actionmonadexamplestrict}
\end{mydefinition}

%

We will further elaborate on bigs-monoidality later in this paper, and we start by presenting two straightforward instances.

\begin{mydefinition}
    A \textbf{bialgebraic category} is a bigs-monoidal category $(\mC, \otimes, I)$ such that the following equalities hold
    \ctikzfig{cipr1eq}
	\ctikzfig{cipr3eq}
\end{mydefinition}

\begin{myremark}
Bialgebraic categories have been investigated in the flownomial calculus
for flowchart description proposed by G. \c{S}tef\u{a}nescu, see e.g. the axioms for angelic branching in \cite{NAGheorghe}
and the references therein.
\end{myremark}

\begin{myremark}
\label{isoToI}
If a bialgebraic category satisfies the law
\ctikzfig{hopf}
 then each object carries the structure of an Hopf algebra~\cite{hopf}, where the antipode for an object $X$ is given by 
 $\id_X$.  
Note that this property does not imply that such a category is special or,
in this case equivalently, connected.
In fact, if a bialgebraic category is connected then 
both $\cobang_X\circ !_X= \id_X$ and 
$!_X\circ \cobang_X= \id_I$ hold, 
that is, each object is isomorphic to $I$.
\end{myremark}

\begin{myexample}\label{example: bialgebra}
	Looking again at Example \ref{example: span x + are also cogs}, we have that the bigs-monoidal 
	categories of spans $(\mathbf{Span}(\mathcal{A}), +,0)$ and relations $(\mathbf{Rel}(\mathcal{A}), +,0)$
	are bialgebraic categories, and in fact they are bicartesian (i.e. both cartesian and cocartesian) monoidal.
	Neither is connected, but the latter is special, the former is not.
	
\end{myexample}

\begin{mydefinition}
	A \textbf{Frobenius category} is a bigs-monoidal category $(\mC, \otimes, I)$ such that the Frobenius law holds, i.e.
	\ctikzfig{frobeniuscorta}
\end{mydefinition}

\begin{myremark}
	The Frobenius law can be used to obtain (see~\cite{pastro2009weak})
	\ctikzfig{frobeniusconsequence}
	and if moreover an object is special~\cite{Bruni2003}
	\ctikzfig{frobeniusconsequence2}
\end{myremark}

\begin{myexample}\label{example: frobenius span}
	Looking once more at Example \ref{example: span x + are also cogs}, we have that the bigs-monoidal 
	categories of spans $(\mathbf{Span}(\mathcal{A}), \times,1)$ and relations $(\mathbf{Rel}(\mathcal{A}), \times,1)$ are Frobenius categories.
	Neither is connected, but both are special.
\end{myexample}

\begin{myremark}
Frobenius categories are also compact-closed (sometimes called 
rigid symmetric monoidal categories), where the dual of an object $X$ is 
the object itself and the unit and counit arrows are defined as $\nabla_A \circ \cobang_A: I \to A \otimes A$ and $\Delta_A \circ !_A: A \otimes A \to I$,
respectively, while the triangle identities trivially hold. 
\end{myremark}

\begin{myremark}
	The Frobenius law is incompatible with the bialgebraic structure: A category that is both Frobenius and bialgebraic
	is also connected, thus 
	each object $X$ 
	is isomorphic to $I$ (see Remark~\ref{isoToI}). Indeed, we have 
	\ctikzfig{frob_bialg_banale}
where the first equality is obtained through the bigs-monoidal structure, the second through the Frobenius law and the last two
through bialgebra equalities. 
\end{myremark}



\subsection{On Markov and restriction categories}

Two categorical notions have come to the forefront in recent years: Markov categories~\cite{Fritz_2020} as models for probabilistic systems, and
restriction categories~\cite{Cockett02,Cockett03,Cockett07} as abstraction of partial functions. This section establishes their correspondence with the notions introduced 
in Sections~\ref{share},~\ref{garbage}, and~\ref{gs-mon}.

\subsubsection{Markov categories}
\begin{mydefinition}
	A \emph{Markov} category is a gs-monoidal category $(\mC, \otimes, I)$ such that for each object $X$ the arrow $!_X: X \to I$ is a terminal morphism, i.e. for each arrow $f:X\to I$ it holds that $f=!_X$.
\end{mydefinition}

\begin{myproposition}
	GS-monoidal categories with projections  correspond exactly to Markov categories.
\end{myproposition}

\begin{proof}
Straightforward, since the presence of projections boils down to the unit $I$ being the terminal object (see Remark~\ref{terminal}), which
in fact is a precise characterisation of Markov categories (see \cite{Fritz_2020}).
\end{proof}

\subsubsection{Restriction categories}
\begin{mydefinition}\label{dfn: restriction category completa}
	A \textbf{restriction structure} on a category  $\mC$ is an assignment which sends every arrow $f:X\to Y$ of $\mC$ to an arrow $\overline{f}:X\to X$ such that the following conditions hold
		
\begin{description}
	\item[(R.1)] $f \circ \overline{f}=f$,
	\item[(R.2)] $\overline{f} \circ \overline{g} = \overline{g} \circ \overline{f}$ for $g: X \to W$,
	\item[(R.3)] $\overline{g \circ \overline{f}}= \overline{g} \circ \overline{f}$ for $g: X \to W$,
	\item[(R.4)] $\overline{g} \circ f= f\circ \overline{g \circ f}$ for $g: Y \to W$.
\end{description}
A \textbf{restriction category} is a category equipped with a restriction structure.

A restriction category has \textbf{restriction terminal object} if there exists an object $I$ such that for every $X\in\mC$ there is a total arrow $t_X:X\to I$ such that $t_I=id_I$ and for every $f:X\to Y$ we have $t_Y \circ f=t_X \circ \overline{f}$.

A restriction category has \textbf{restriction binary products} if there exist a restriction functor 
$-\times -: \mC\times \mC\to \mC$
(i.e.\ such that $\overline{f\times g}=\overline{f}\times \overline{g}$) and total arrows $\Delta: X\to X\times X$, $p:X\times Y\to X$ and $q:X\times Y\to Y$ satisfying

\[\begin{tikzcd}[row sep=10pt]
	& X &&& {X\times Y} \\
	X & {X\times X} & X && {X\times Y\times X\times Y} & {X\times Y} \\
	 \\ 
	{X\times Y} & {X\times Y} & {X\times Y} && X & X \\
	X && Y &&& {X\times X} \\
	{X'} & {X'\times Y'} & {Y'} && {X'} & {X'\times X'}
	\arrow["id"', from=1-2, to=2-1]
	\arrow["\Delta"', from=1-2, to=2-2]
	\arrow["id", from=1-2, to=2-3]
	\arrow["\Delta"', from=1-5, to=2-5]
	\arrow["id", from=1-5, to=2-6]
	\arrow["p", from=2-2, to=2-1]
	\arrow["q"', from=2-2, to=2-3]
	\arrow["{p\times q}"', from=2-5, to=2-6]
	\arrow["p"', from=4-1, to=5-1]
	\arrow["{\overline{f}\times\overline{g}}"', from=4-2, to=4-1]
	\arrow["{\overline{f}\times\overline{g}}", from=4-2, to=4-3]
	\arrow["{f\times g}"', from=4-2, to=6-2]
	\arrow["q", from=4-3, to=5-3]
	\arrow["{\overline{f}}", from=4-5, to=4-6]
	\arrow["f"', from=4-5, to=6-5]
	\arrow["\Delta", from=4-6, to=5-6]
	\arrow["f"', from=5-1, to=6-1]
	\arrow["g", from=5-3, to=6-3]
	\arrow["{f\times f}", from=5-6, to=6-6]
	\arrow["p", from=6-2, to=6-1]
	\arrow["q"', from=6-2, to=6-3]
	\arrow["\Delta"', from=6-5, to=6-6]
\end{tikzcd}\]
A \textbf{cartesian restriction category} is a restriction category equipped with restriction terminal object and restriction binary products.
\end{mydefinition}
\begin{myproposition}\label{gs as restriction}
	GS-monoidal categories with diagonals  correspond exactly to cartesian restriction categories.
\end{myproposition}

\begin{proof}
Moving from gs-monoidal categories to restriction categories, the core of the proof occurs in \cite[Rem. 2.15]{FritzGCT23}, where gs-monoidal categories are shown to have enough structure to define the 
\textit{domain} of an arrow $f:X\to Y$ as
	\ctikzfig{domain}
Then, taking $\bar{f}\coloneqq \mathrm{dom}(f)$, the axioms ($R.2$) and ($R.3$) of Definition~\ref{dfn: restriction category completa}
actually hold for any gs-monoidal category, while ($R.1$) and ($R.4$) hold for the presence of diagonals.
The proof that they have restriction products is a straightforward check. We refer to \cite{FritzGCT23,Cockett07} and the references therein for all the details.
\end{proof}

\begin{myremark}
	The previous proposition can be dualised, providing the characterisation of \emph{corestriction categories}, see \cite[Ex. 2.1.3(12)]{Cockett02}, in terms of cogs-monoidal categories.
\end{myremark}
\begin{myremark}
An alternative proof of Proposition~\ref{gs as restriction} can be found in~\cite[Ex. 2.1.3(6-7)]{Cockett02} and \cite[Sec. 4]{Cockett07}.
Note that cartesian restriction categories have been proved equivalent to several other notions, such as  p-categories with one-element object~\cite{Robinson88}, 
partial cartesian categories in the sense of Curien and Obtulowicz \cite{CURIEN198950}, and they are a special case of the bicategories of partial maps of Carboni~\cite{Carboni_87}.
See also Proposition~\ref{oplax as restriction} later.
\end{myremark}

\begin{myremark}
 In \cite{cho_jacobs_2019}, the authors  independently introduced \textit{CD categories},
 which are the same as gs-monoidal categories and they used the term \textit{affine} CD categories for what are today known as Markov categories. The term Markov category was introduced in \cite{Fritz_2020}. In the context of Markov categories, functional arrows are referred to as \textit{deterministic}.
\end{myremark}


\section{A taxonomy of Kleisli categories}
\label{sec:taxonomy kleisli}

Our taxonomy of Kleisli categories is done for symmetric monoidal monads, which are equivalent to commutative monads, see \cite{Kock72} and now \cite[Appendix~C]{Fritz_2021}. 
We discuss three classes of such monads: affine and relevant monads, first considered in \cite{Kock71,Jacobs1994}, and gs-monoidal monads (which are both affine and relevant). 

We then study the properties of their Kleisli categories depending on the categorical structures we previously considered. Indeed, rephrasing \cite[Thm. 4.3]{Jacobs1994}, the Kleisli category of a relevant/affine monad on a cartesian category is respectively a restriction/Markov category. Hence, we extend this result considering affine and relevant monads on our 
taxonomy. 

For instance, in the following diagram  we represent for a relevant monad the corresponding nature of its Kleisli category using a dashed arrow. Hence, we prove that the Kleisli category of a relevant 
monad on a cartesian restriction category is a cartesian restriction category, while the Kleisli category of a relevant monad on a Markov category is just gs-monoidal.

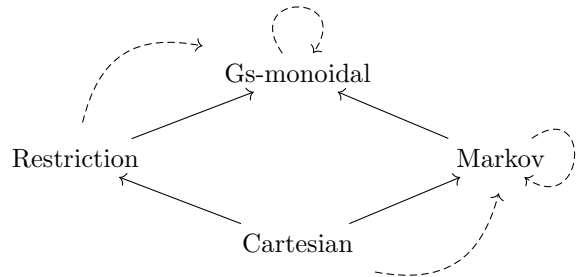
\[
\begin{tikzcd}
	& {\text{GS-monoidal}} \\
	{\text{Markov}} && {\hspace{-2cm}\text{Cartesian restriction}} \\
	& {\text{Cartesian monoidal}}
	\arrow[dashed, from=1-2, to=1-2, loop, in=55, out=125, distance=10mm]
	\arrow[from=2-1, to=1-2]
	\arrow[shift left=3, curve={height=-18pt}, dashed, from=2-1, to=1-2]
	\arrow[from=2-3, to=1-2]
	\arrow[dashed, from=2-3, to=2-3, loop, in=325, out=35, distance=10mm]
	\arrow[from=3-2, to=2-1]
	\arrow[from=3-2, to=2-3]
	\arrow[shift right=3, curve={height=18pt}, dashed, from=3-2, to=2-3]
\end{tikzcd}\]

A similar taxonomy for the enriched case is going to be provided in Section~\ref{sec:oplax cartesian categories}, 
where we consider the Kleisli categories of enriched monads that are \textit{colax relevant} and \textit{colax affine}, as well as
 \textit{colax gs-monoidal}.

\begin{mydefinition}[Symmetric monoidal monad]\label{dfn: symmetric monoidal monad}
	Let $\mathcal{C}$ be a symmetric monoidal category. Let $T : \mathcal{C} \to \mathcal{C}$ be a monad carrying the structure of a lax symmetric monoidal functor with structure maps $\freccia{\otimes \circ \, (T\times T)}{c}{T\circ \otimes}$ and $u : I \to TI$. Then $T$ is a \textbf{symmetric monoidal monad} if $u = \eta_I$ and the following two diagrams commute
\begin{equation}\label{diagram: symmetric monoidal monad 1}
\begin{tikzcd}
& X \otimes Y \arrow[dl, "\eta \otimes \eta"'] \arrow[dr, "\eta"] & \\
TX \otimes TY \arrow[rr, "c"'] & & T(X \otimes Y)
\end{tikzcd}
\end{equation}

\begin{equation}\label{diagram: symmetric monoidal monad 2}
	\begin{tikzcd}
TTX \otimes TTY \arrow[r, "c"] \arrow[d, "\mu \otimes \mu"'] & T(TX \otimes TY) \arrow[r, "Tc"] & TT(X \otimes Y) \arrow[d, "\mu"] \\
TX \otimes TY \arrow[rr, "c"'] & & T(X \otimes Y)
\end{tikzcd}
\end{equation}

\end{mydefinition}
\begin{myremark}
\label{rem:comm_vs_symmon}
	It is well-known that on a symmetric monoidal category,  symmetric monoidal monads are equivalent to commutative monads, see \cite[Th.~2.3]{Kock72} and \cite[Th.~3.2]{Kock70}. Definition \ref{dfn: symmetric monoidal monad} corresponds to the notion of monad internal to the 2-category of symmetric monoidal categories, lax functors and monoidal natural transformations. The commutativity of diagrams (\ref{diagram: symmetric monoidal monad 1}) and (\ref{diagram: symmetric monoidal monad 2}) means that $\mu$ and $\eta$ are monoidal natural transformations.
\end{myremark}

\begin{myexample}\label{monadM}
Let us consider the lax symmetric monoidal functor $\mathcal{M}$ introduced in Section~\ref{WR}.
This functor defines a symmetric monoidal monad on $(\mathbf{Set}, \times, \{\bullet\})$,
sometimes called the \emph{semiring monad}, 
with natural transformations $\eta_X: X\to \mathcal{M}(X)$ and $\mu_X:\mathcal{M}(\mathcal{M}(X))\to\mathcal{M}(X)$ given by
\begin{itemize}
	\item[-] \(\eta_X(x_0)(x)=\begin{cases}
		1\ \text{if}\ x=x_0\\
		0\ \text{otherwise}
	\end{cases}\)
	\item[-] $\mu_X(\lambda)(x)=\underset{h\in\mathcal{M}(X)}{\bigoplus}\lambda(h)\cdot h(x)$
	\end{itemize}

\end{myexample}

\begin{myremark}\label{weightMon}
	Notice that, in the previous example, if $M$ is the semiring of Booleans, then  $\mathcal{M}$ is the finite subsets monad $\mathcal{P}$, while 
	if $M$ is the semiring of natural numbers $\mathbb{N}$, then $\mathcal{M}$ is the 
	monad of \textit{finite multisets}, see \cite{Gadducci08}. 
	Moreover, if $M$ is the semiring of positive real numbers $\mathbb{R}^+$, then $\mathcal{M}$ is the monad of \textit{unnormalised probability distributions}, see \cite{cho_jacobs_2019}, while if $M$ is either the semiring $([0,1],\text{max},\text{min}, 0, 1)$ or the semiring $([0,1],\text{max},\cdot, 0, 1)$ one obtains Golubtsov's categories of 
	\textit{fuzzy information transformer}, see~\cite{golubtsov2002monoidal}. 
\end{myremark}

\begin{mydefinition}\label{dfn: various monads}
	A symmetric monoidal monad $(T,\mu, \eta,c, u)$  on a gs-monoidal category $\mC$ is
	\begin{itemize}
	\item \textbf{affine} if $(T,c,u)$ is an affine functor;
	 \item \textbf{relevant} if  $(T,c,u)$ is a relevant  functor;
	\item \textbf{gs-monoidal} if $(T,c,u)$ is a gs-monoidal functor.
	\end{itemize} 
\end{mydefinition}

\begin{myremark}
	Affine and relevant monads were introduced in \cite{Kock71} in the context of cartesian monoidal categories. While the term ``affine'' is used in \cite{Kock71}, the term ``relevant'' appears in \cite{Jacobs1994} due to a connection with \textit{relevant logic}.
In particular, in \cite[Th.~2.1]{Kock71} the author considers as definition of affine monads on a cartesian monoidal category one of the two equivalent conditions
\begin{itemize}
	\item the unit of the monad $\eta_I:I\to T(I)$ is an isomorphism
	\item the following diagram commutes
	\[
	\begin{tikzcd}
		{T(X)\otimes T(Y)} & {T(X\otimes Y)} \\
		& {T(X)\otimes T(Y)}
		\arrow["{\psi_{X,Y}}", from=1-1, to=1-2]
		\arrow["id"', Rightarrow, no head, from=1-1, to=2-2]
		\arrow["{\langle T(\pi_1),T(\pi_2)\rangle}", from=1-2, to=2-2]
	\end{tikzcd}\] 
\end{itemize}
Indeed, if the category $\mC$ is cartesian, then $I$ is terminal. This fact, together with the commutativity of  diagram (\ref{diagram: lax affine}) of Definition~\ref{def garbage functor} for $X=I$, which states that $\eta_I \circ !_{T(I)}=\id_{T(I)}$, 
implies that $\eta_I$ is an isomorphism. 
Hence, Definition~\ref{dfn: various monads} generalises that of affine monads for cartesian monoidal categories since it can be applied in contexts in which $I$ is not terminal, as in gs-monoidal categories.

For relevant monads, Kock proves in \cite{Kock71} that the condition for relevant functors (diagram (\ref{diagram: lax relevant}) of Definition~\ref{def relevant functor}) 
is equivalent to the commutativity of the diagram
\[
\begin{tikzcd}[column sep= huge]
	{T(X\otimes Y)} & {T(X)\otimes T(Y)} \\
	& {T(X\otimes Y)}
	\arrow["{\langle T(\pi_1),T(\pi_2)\rangle}", from=1-1, to=1-2]
	\arrow["id"', Rightarrow, no head, from=1-1, to=2-2]
	\arrow["{\psi_{X,Y}}", from=1-2, to=2-2]
\end{tikzcd}\]
see \cite[Prop.~2.2]{Kock71} and also \cite[Lem.~4.2]{Jacobs1994}.
\end{myremark}

\begin{myremark}
 Observe that for any object $X$ of a gs-monoidal category, the hom-set $\mC(X,I)$ has canonically 
 the structure of a commutative monoid. The multiplication
 \[\mC(X,I)\times \mC(X,I)\to \mC(X,I)\]
 sends a pair of arrows $(f,g)$ to $\lambda_I \circ (f\otimes g)\circ \nabla_X$, while the unit is given by $!_X$.
 This simple observation can be used to generalise the ordinary notion of affine categories and monads, by requiring additional algebraic properties for such a monoid. Indeed, the more general notions \emph{weakly Markov} category and \emph{weakly affine} monad 
 have been introduced in~\cite{FritzGPT23}: Weakly Markov categories are gs-monoidal categories such that the monoid $\mC(X,I)$ is a group. For the corresponding weakly affine monads,
 instead of assuming $T(I)$ to be isomorphic to $I$,  it is just required that the commutative monoid structure of $T(I)$ is a group.
 We are not aware of a diagrammatic characterisation as for affine categories and monads, making it applicable to gs-monoidal and relevant categories.
\end{myremark}

\begin{myremark}\label{rmk: funtore C ---> C_T}
	It is well-known that a monad $(T,\mu,\eta,c,u)$ on a symmetric monoidal category $\mC$ induces a 
	symmetric monoidal structure on $\mC_T$ if the monad is symmetric monoidal. In particular, 
	the tensor product on $\mC_T$ is defined as $X\otimes^{\sharp} Y:= X\otimes Y$ and $f\otimes^{\sharp} g:= c_{X',Y'}\circ(f\otimes g)$, where $f:X\to X'$ and $g:Y\to Y'$ are arrows of $\mC_T$.
	The obvious functor $\mathcal{K}:\mC\to\mC_T$, which is the identity on objects and acts by post-composition with $\eta$ on the arrows, is strict symmetric monoidal
\[\otimes^{\sharp}\circ(\mathcal{K}\times \mathcal{K})= \mathcal{K}\circ \otimes\]
and the arrows defining the symmetric monoidal structure on $\mC_T$ are obtained as the image of those defining the structure of $\mC$ through the functor $\mathcal{K}$. 

Hence, it is possible to conclude that $\mathcal{K}$ preserves equalities of arrows
$f=g$ of $\mC$, where $f$ and $g$ are obtained through compositions and products of the structural arrows of $\mC$. 
\end{myremark}


  The above argument was already observed in \cite{FritzGCT23} (see also \cite[Lem.~2]{jacobsdion}) and it is the key step in proving the following result given in \cite[Prop.~4.4]{FritzGCT23}.

  \begin{myproposition}\label{prop: Kleisli FGTC}
  Let $(T,\mu, \eta,c, u)$ be a symmetric monoidal monad on a gs-monoidal category $\mC$. Then the Kleisli category $\mC_T$ is a gs-monoidal category with 
   $\nabla_X$ and $!_X$ given for every object $X$ by 
  \[\nabla^\sharp_X:=\eta_{X\otimes X} \circ \nabla_X,\qquad !^\sharp_X:=\eta_I \circ !_X.\] 
  \end{myproposition}

  In other words, $\mathcal{K}:\mC\to \mC_T$ is a (strict) gs-monoidal functor. 
  
  \begin{myexample}\label{exPesRelGS}  The Kleisli category $\mathbf{Set}_{\mathcal{M}}$ of the semiring monad $\mathcal{M}$ introduced in Section~\ref{WR} is the category of sets
	and \textit{weighted relations}, which is a gs-monoidal category. 
	Some instances for different semirings $M$ are presented in Remark~\ref{weightMon}.
  \end{myexample}
  
  We also obtain a further instance of the above result.

\begin{myproposition}\label{prop: Kleisli of affine monads}
    Let $(T,\mu, \eta,c, u)$ be an affine monad on a gs-monoidal category $\mC$. If $\mC$ has projections then so does the Kleisli category $\mC_T$.
	\begin{proof}
If $I$ is terminal and $T$ is affine, then $T(I)\cong I$. Hence, $I$ is terminal also in $\mC_T$. 
    \end{proof}
\end{myproposition}

\begin{myexample}
  As shown in Section~\ref{WR}, the functor $\mathcal{M}_u$ is affine, hence the Kleisli category $\mathbf{Set}_{\mathcal{M}_u}$ has projections. 
  For the Boolean semiring, the Kleisli category of $\mathcal{P}_u$ in Remark~\ref{power}
	is the category whose objects are sets and whose arrows are total relations.
%
\end{myexample}

\begin{myremark}
As an instance of Proposition~\ref{prop: Kleisli of affine monads}, we have that if $\mC$ is a cartesian monoidal category and $T$ is affine, then the Kleisli category $\mC_T$ has projections.
Note instead that if $\mC$ has diagonals, then $\mC_T$ does not necessarily have projections. Indeed, the identity functor is an affine monad.
\end{myremark}
\begin{myproposition}\label{prop: Kleisli of relevant monads}
    Let $(T,\mu, \eta,c, u)$ be a relevant monad on a gs-monoidal category $\mC$. If $\mC$ has diagonals then so does the Kleisli category $\mC_T$.
	\begin{proof}
		We need to prove the naturality of $\nabla^\sharp$, i.e. for every arrow $f:X\rightarrow Y$ in $\mC_T$, which corresponds to an arrow $f:X\to T(Y)$ in $\mC$, we show that
				\[\nabla^\sharp_Y\circ^\sharp f = (f\otimes^\sharp f)\circ^\sharp \nabla^\sharp_X \] in $\mC_T$.
				Since the composition of two arrows $f:X\rightarrow Y$ and $g:Y\to Z$ in $\mC_T$ is obtained as $\mu_Z\circ T(g)\circ f$, then we obtain
				\begin{align}
					(f\otimes^\sharp f)\circ^\sharp \nabla^\sharp_X &= \mu_{Y\otimes Y}\circ T(c)\circ T(f\otimes f)\circ\eta_{X\otimes X}\circ\nabla_X \notag\\
					&= \mu_{Y\otimes Y}\circ \eta_{T(Y\otimes Y)}\circ c_{Y,Y}\circ (f\otimes f)\circ\nabla_X \tag{naturality of $\eta$}\\
					&= \mu_{Y\otimes Y}\circ \eta_{T(Y\otimes Y)}\circ c_{Y,Y}\circ \nabla_{TY}\circ f \tag{naturality of $\nabla$}\\
					&= \mu_{Y\otimes Y}\circ \eta_{T(Y\otimes Y)}\circ T(\nabla_Y)\circ f  \tag{relevant monad}\\
					&= \nabla_Y^\sharp \circ^\sharp f. \tag*{\qedhere}
				\end{align}
			\end{proof}
\end{myproposition}

\begin{myexample}\label{exPesRel}
	The Kleisli category of the relevant monad $\mathcal{P}_e$ discussed in Remark~\ref{power} is the category of sets
	  and partial functions, which has diagonals. 
	\end{myexample}

\begin{myremark}
As an instance of Proposition~\ref{prop: Kleisli of relevant monads}, we have that if $\mC$ is a cartesian monoidal category and $T$ is relevant, 
then the Kleisli category $\mC_T$ has diagonals.
Note instead that if $\mC$ has projections, then $\mC_T$ does not necessarily have diagonals.
Indeed, the identity functor is a relevant monad.
\end{myremark}

This result has been proved in \cite[Lem.~4.11]{FritzGCT23},
and it is now a consequence of the two propositions given above.

\begin{mycorollary}
    Let $(T,\mu, \eta,c, u)$ be a gs-monoidal monad on a gs-monoidal category $\mC$. If $\mC$ is cartesian monoidal then so is the Kleisli category $\mC_T$. 
\end{mycorollary}


\begin{myexample}
\label{CSMonad}
	Let $(\mathcal{C}, \otimes, I)$ be a gs-monoidal category and $M\in\mathcal{C}$ equipped with a monoid structure. 
	Then the \emph{action monad} $(-)\otimes M:\mathcal{C}\to \mathcal{C}$ is a symmetric 
	monoidal monad with $u = \cobang_M$ and
	$c_{X,Y}:(X\otimes M) \otimes (Y\otimes M)\to (X\otimes Y)\otimes M$
	given by 
	\ctikzfig{ex3_19}
If $M$ is connected (see Definition \ref{dfn: specal and connected bimonoid}),  then the monad is affine. If $M$ is special
(see again Definition \ref{dfn: specal and connected bimonoid}), then the monad is relevant.
\end{myexample}


We close by dualising Proposition~\ref{prop: Kleisli FGTC} and its consequences.

\begin{myproposition}\label{prop:Kleisli cogs}
	Let $(T,\mu, \eta,c, u)$ be a symmetric monoidal monad on a cogs-monoidal category $\mC$. Then the Kleisli category $\mC_T$ is a cogs-monoidal category with $\Delta_X$ and 
	$\cobang_X$ given for every object $X$ by 
\[\Delta^\sharp_X:= \eta_{X} \circ \Delta_X,\qquad \cobang^\sharp_X:=\eta_X \circ \cobang_X\]
\end{myproposition}


\begin{mycorollary}\label{cor: kleisli bigs}
	Let $(T,\mu, \eta,c, u)$ be a symmetric monoidal monad on a bigs-monoidal (or bialgebraic or Frobenius) category $\mC$. Then so is the Kleisli category $\mC_T$.
\end{mycorollary}

\begin{myremark}
As noted, a cogs-monoidal category is the dualisation of a gs-monoidal category. The same dualisation can be exploited by introducing coprojections and codiagonals, as well as
coaffine and corelevant functors.
The results  about Kleisli categories in this section can be simply restated for these notions. We leave them as well as their enriched versions
in Appendix \ref{section: appendix dual results}, which also describes the dual of the results 
obtained in forthcoming Section \ref{subsection:enriched taxonomy}.
\end{myremark}

\section{A taxonomy of the enriched context}\label{sec:oplax_cart}
\label{sec:oplax cartesian categories}

The 2-categorical formalism 
has been useful in the context of term and graph rewriting, where rewriting sequences are modelled as 2-cells (for early takes see e.g. \cite{GadducciHL99} and the references therein). 
In this work we consider 
preorder-enriched categories,
which are well-suited in categories for relations, since they 
have a canonical ordering given by set-theoretic inclusion
(while this enrichment becomes trivial if we consider functions instead of relations). 
In other words, functions are inherently one-dimensional, while relations provide a more complex view, which can be described through a richer categorical structure.
We find two relevant families of 2-categories that abstract the poset-enriched structure of $\Rel$: Cartesian bicategories \cite{CARBONI198711,cartesianbicatII} 
by Carboni and Walters and allegories \cite{freyd1990categories} by Freyd and Scedrov.
\[
\begin{tikzcd}[column sep= tiny]
	& {\text{oplax cartesian }} \\
	& {\text{oplax bicartesian} } \\
	& {\text{cartesian bicategories}} \\
	{\text{bicategories of relations}} && {\text{bicategories of bialgebras}} \\
	&  \mathbf{1}
	\arrow[from=2-2, to=1-2]
	\arrow[from=3-2, to=2-2]
	\arrow[from=4-1, to=3-2]
	\arrow[from=4-3, to=3-2]
	\arrow[from=5-2, to=4-1]
	\arrow[from=5-2, to=4-3]
\end{tikzcd}\]
Hence, in the enriched context, the notion of gs-monoidal category is replaced by the notion of \textit{preorder-enriched gs-monoidal category}. 
Requiring the oplax naturality of duplicator $\nabla_X$ and of discharger $!_X$, one obtains the notion of \textit{oplax cartesian category}. 
%
%
If one requires both a comonoid and monoid structure on the objects and their oplax naturality, one obtains the notion of \textit{oplax bicartesian category}. 
Cartesian bicategories can be then presented as oplax bicartesian categories in which the enrichment is posetal and satisfies three additional conditions.
%
Special cases of cartesian bicategories are bicategories of relations (those satisfying \textit{Frobenius law}) and bicategories of bialgebras. Note that these two structures are incompatible: 
a bicategory of relations  that is also a bicategory of bialgebras has to be the trivial one.

\subsection{Order-enriched categories}
We have seen that gs-monoidal categories enjoy some features of $\Rel$ with respect to total and functional arrows. 
Our next step is to recall how to build on the notion of gs-monoidality to account for the usual poset-enrichment of $\Rel$. Recall that a \textit{preorder} is a set $X$ equipped with a reflexive and transitive relation $\leq$. A \textit{poset} is a preorder such that $\forall x,y\in X\ (x\le y \wedge y\le x \Rightarrow x=y)$. The category of preorders and monotone functions is denoted by $\mathbf{PreOrd}$, while the category of posets and monotone functions is denoted by $\mathbf{PO}$. The inclusion $\mathbf{PO}\to \mathbf{PreOrd}$ has a left adjoint, called \textit{poset-reflection}, which maps a preorder to the corresponding poset obtained by quotienting the equivalence relation $x\sim y$ if $x\le y$ and $y\le x$.

\begin{mydefinition}\label{def poset-enriched gs-monoidal category}
	A \textbf{preorder-enriched gs-monoidal category} $\mC$ is a gs-monoidal category $\mC$ that is also a preorder-enriched monoidal category.
\end{mydefinition}

In the following we will often consider also poset-enriched categories. 
Recall that a preorder-enriched monoidal category consists of a preorder-enriched category $\mC$, an object $I$ of $\mC$, 
a preorder-enriched functor $\freccia{\mC\times \mC}{\otimes}{\mC}$,  and enriched monoidal natural isomorphisms
\[
	\freccia{I\otimes -}{\lambda}{\id_{\mC}}, \qquad \freccia{- {}\otimes{} I}{\rho}{\id_{\mC}}, \qquad \freccia{(-\otimes -)\otimes -}{\alpha}{-\otimes (-\otimes -)}
\]
such that the underlying category equipped with the underlying functor $\otimes$, the object $I$, and the natural isomorphisms $\lambda$, $\rho$, and $\alpha$ is a monoidal category (see \cite{Kelly05} for details).
Since a preorder-enriched functor is just an ordinary functor that is in addition monotone, the preorder structure and the monoidal structure are required to interact by the monotonicity of the tensor product $\otimes$; the preorder-enrichment of the structure isomorphisms $\lambda$, $\rho$, and $\alpha$ does not add any additional condition since preorder-enrichment for natural transformations between preorder-enriched functors is trivial.
\begin{myexample}\label{example: span is poset-enriched}
Let us consider the category of spans  $\mathbf{Span}(\mA)$  presented in Example~\ref{ex_spans_are_gs}. It is well-known that this category has a natural 2-categorical structure, where the 2-cells are defined as follows
\begin{itemize}
\item a 2-cell $\alpha: (X \leftarrow A \to Y) \Rightarrow (X \leftarrow  B \to Y)$ is an arrow $\alpha: A \to B$ in ${\mA}$  such that the following diagram commutes
		\[
			\begin{tikzcd}[row sep=2pt, column sep=8pt]
				& A \ar[dl, bend right] \ar[rd, bend left]\ar[dd, "\alpha"] & \\
				X  & & Y  \\
				& B \ar[ul, bend left] \ar[ur,  bend right]
			\end{tikzcd}
		\]
	\item vertical composition of 2-cells is given by composition in $\mA$;
	\item horizontal composition of 2-cells as well as associators and unitors are induced by the universal property of pullbacks.
	 \end{itemize}
We will refer to $\mathbf{PSpan}(\mathcal{A})$
as the preorder-enriched category obtained by considering the preorder-reflection of the previous 2-categorical structure,
and to $(\mathbf{PSpan}(\mathcal{A}),\times,1)$ and $(\mathbf{PSpan}(\mathcal{A}),+,0)$ as its gs-monoidal counterparts.
Given the correspondence in Remark~\ref{spas as rel}, for a regular category $\mA$ in which every regular epi splits the poset-reflection of the previous 2-category boils down 
to what we call $\mathbf{PRel}(\mathcal{A})$, i.e. the usual poset-enrichment of $\mathbf{Rel}(\mathcal{A})$.
\end{myexample}

\begin{myexample}\label{SetM_preordered}
	Recall that a semiring $(M,\oplus,\odot,0,1)$  comes equipped with a canonical preorder $\leq_M$, namely $a \leq_M b$ if there exists $c$ such that $a \oplus c = b$,
	so the Kleisli category of the semiring monad $\mathcal{M}$ introduced in Example \ref{example: monad semiring} is 
	canonically a preorder-enriched gs-monoidal category. It is gs-monoidal thanks to Proposition \ref{prop: Kleisli FGTC} and is preorder-enriched,
	assuming that
	for two arrows $f,g:X \to Y$ in $\mathbf{Set}_{\mathcal{M}}$,  $f\le g$ if $\forall x\in X\forall y\in Y\  f(x)(y)\le_M g(x)(y)$.
	
	Note that the same construction occurs if we have an ordered semiring $(M,\oplus,\odot,0,1)$, i.e. a semiring such that $(M, \leq)$ is a preorder
	and $+$ and $\cdot$ are functions that are order-preserving on both arguments.
\end{myexample}

\begin{myexample}
\label{preord}
The category $\mathbf{PreOrd}$ is cartesian monoidal, and the gs-monoidal structure is inherited from the direct product in $\mathbf{Set}$.
It also comes equipped with a canonical preorder-enrichment, assuming that
for two arrows $f,g: X \to Y$ in $\mathbf{PreOrd}$ we have that $f\le g$ if $\forall x\in X\  f(x)\le g(x)$ in $Y$.
\end{myexample}


In a general preorder-enriched gs-monoidal category, no further compatibility with a gs-monoidal structure is required.
However, most often additional axioms hold.
We recall the notion of \emph{oplax cartesian category} \cite[Def.~3.2]{FritzGCT23}. 

\begin{mydefinition}\label{def oplax cartesian cat}
An \textbf{oplax cartesian category} $\mC$ is a preorder-enriched gs-monoidal category $\mC$ such that every arrow is \textbf{oplax copyable} and \textbf{oplax discardable}, 
i.e. the following inequalities hold for every arrow $\freccia{X}{f}{Y}$

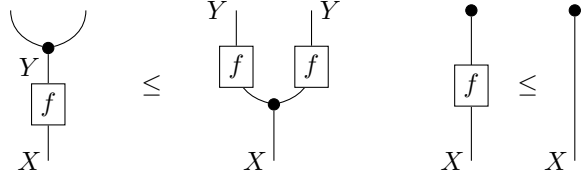
	\ctikzfig{oplax_cart_def}

\end{mydefinition}
The notion of \emph{oplax cocartesian category} can be given as expected, starting from a cogs-monoidal category and
preserving the direction of inequalities. Similarly,
one gets the notion of \emph{oplax bicartesian category}
if both previous notions hold.
%
In the following, we premise the adjective \emph{posetal} if the underlying order is a poset.

\begin{myexample}\label{example: pspan e' oplax}
	The category $(\mathbf{PSpan}(\mathcal{A}), \times, 1)$ that was introduced in Example \ref{example: span is poset-enriched} 
	is oplax cartesian~\cite[Prop. 5.3]{FritzGCT23},
	while $(\mathbf{PRel}(\mathcal{A}), \times, 1)$ is poset-enriched. In fact, both are oplax bicartesian.
	As for $\mathbf{PreOrd}$ with the direct product, since it is cartesian monoidal, it is also
	oplax cartesian.
\end{myexample}

\begin{myremark}
	As it occurs for gs-monoidal categories, also the notion of oplax cartesian category could be split in two parts,
	that is, we could consider separately the arrows that are \emph{oplax copyable} 
	and those that are \emph{oplax discardable}, as in the previous sections.
	For the sake of readability, in this case we preferred to keep these two notions together, so that we are leaving implicit 
	the obvious partition and generalisation of the forthcoming results, such as Proposition~\ref{prop: Kleisli colax cartesian}.
\end{myremark}

\begin{myexample}
\label{sub}
	Let ${\mathcal{M}}: \mathbf{Set} \to \mathbf{Set}$ be the functor introduced in Section~\ref{WR} for semiring $(M,\oplus,\odot,0,1)$  equipped with a preorder such as in Example \ref{SetM_preordered}. 
	Consider the following slight modifications of $\mathcal{M}$
	\[\mathcal{M}_e^s(X)=\left\{ h:X \to M \ |\  h\  \text{has support at most one and is sub-idempotent}\right\}\]
	\[\mathcal{M}_u^s(X)=\left\{ h:X \to M \ |\  h\  \text{has finite support and is sub-normalised}\right\}\]
	where sub-idempotent means that $\forall x \in X.\, h(x)\le h(x) \odot h(x)$ and sub-normalised that $\bigoplus_{x \in X}h(x) \le 1$.
	In the Kleisli category $\mathbf{Set}_{{M}_e^s}$ every arrow is oplax copyable, while in 
	$\mathbf{Set}_{{M}_u^s}$ every arrow is oplax discardable.	
	\end{myexample}

\newpage
\begin{myremark}
\label{laxbialg}
Note that in every bigs-monoidal category that is either oplax cartesian or oplax cocartesian the following inequalities hold

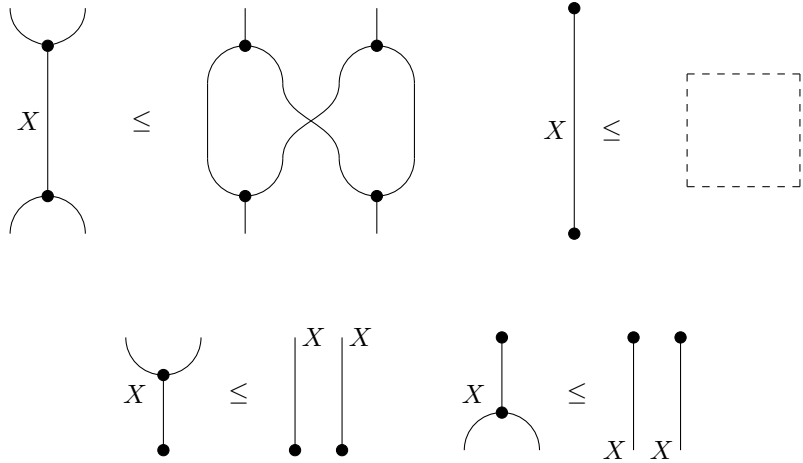
\ctikzfig{cipr1}
\ctikzfig{cipr3}

In other terms, each object has a \emph{lax} (bicommutative) bimonoid structure or, 
equivalently, such categories have a \emph{lax} bialgebraic structure.
\end{myremark}

In the following proposition we observe that the notion of oplax cartesian category may subsume that of restriction category.
\begin{myproposition}\label{oplax as restriction}
Let $\mC$ be a posetal oplax cartesian category such that the following inequality holds
	for every pair of arrows $\freccia{X}{f}{Y}$ and $\freccia{Y}{g}{W}$
\ctikzfig{condition_oplax_is_restriction}
Then $\mC$ is a restriction category.
\end{myproposition}

\begin{proof}
We already noted in the proof of Proposition~\ref{gs as restriction}
that the axioms ($R.2$) and ($R.3$) of Definition~\ref{dfn: restriction category completa} hold for any gs-monoidal category when we consider the restriction structure given by the domain of an arrow
	\ctikzfig{domain}
Now, the proof that ($R.1$) holds follows from \cite[Prop.~3.6]{FritzGCT23}.
As for  ($R.4$), it suffices to use the first axiom of oplax cartesian categories and the additional requirement for $\mC$.
\end{proof}

\begin{myremark}
The previous proposition shows that the notions of posetal oplax cartesian category and restriction category are very closely related, since the former satisfies the first three requirements 
($R.1$)-($R.3$) of the latter,
and the last condition ($R.4$) is recovered by the additional inequality in Proposition~\ref{oplax as restriction}.
Assuming such inequality, the restriction structure induced by an oplax cartesian category satisfies all the necessary conditions 
for having restriction products in the sense of \cite[p. 20]{Cockett07}, except for an inequality of the final condition, namely
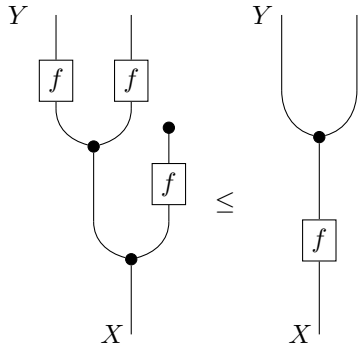
\ctikzfig{rest_products}
It is a simple check, and indeed a sanity check, that in posetal oplax cartesian categories the above condition corresponds via ($R.1$) and coassociativity to $(f\otimes f)\circ \nabla_X\le \nabla_Y\circ f$, hence to the naturality of duplicators.
\end{myremark}

We close the thread with a result concerning positivity, showing that the positivity axiom is stronger than the axiom (R4) for restriction categories.

\begin{mycorollary}
Let $\mC$ be a posetal oplax cartesian category.
If it is positive, then it is a restriction category.
\end{mycorollary}
\begin{proof}
 Recall that by Definition~\ref{positive} a category is positive
    if for every pair of arrows $f:X\to Y$ and $g: Y \to W$ such that 
    $g \circ f$ is functional then
\ctikzfig{positivity}
holds. Now, it is easy to note that in a posetal oplax cartesian category
for any arrow $h: X \to W$ we have that $!_W \circ h$ is functional. In particular, we have that $!_W\circ g \circ f$ is functional, and hence, by positivity, we have that the inequality of Proposition~\ref{oplax as restriction}
 \ctikzfig{condition_oplax_is_restriction}
is satisfied, and hence ($R.4$) follows.
\end{proof}



On functors between oplax cartesian categories, one often has additional inequalities, which take the following form.

	\begin{mydefinition}\label{def:(op) oplax cartesian functor}
		Let $\mC$ and $\mD$
		be preorder-enriched gs-monoidal categories and  
		$\freccia{\mC}{F}{\mD}$ a preorder-enriched lax symmetric monoidal functor with structure arrows $\psi, \psi_0$.
		Then $F$ is called 
			\textbf{colax affine} if the following inequality holds 
\[
\begin{tikzcd}
	{F(X)} && {F(I)} \\
	& I
	\arrow[""{name=0, anchor=center, inner sep=0}, "{F(!_X)}", from=1-1, to=1-3]
	\arrow["{!_{F(X)}}"', from=1-1, to=2-2]
	\arrow["{\psi_0}"', from=2-2, to=1-3]
	\arrow["\le"{anchor=center, rotate=-90}, draw=none, from=2-2, to=0]
\end{tikzcd}
\]
	and it is called \textbf{colax relevant}  if the following inequality holds
	\[
\begin{tikzcd}
	{F(X)} && {F(X\otimes X)} \\
	& {F(X)\otimes F(X)}
	\arrow[""{name=0, anchor=center, inner sep=0}, "{F(\nabla_X)}", from=1-1, to=1-3]
	\arrow["{\nabla_{F(X)}}"', from=1-1, to=2-2]
	\arrow["{\psi_{X,X}}"', from=2-2, to=1-3]
	\arrow["\le"{description,rotate=-90}, draw=none, from=2-2, to=0]
\end{tikzcd}\]
	If $F$ is both colax affine and colax relevant it is called \textbf{colax gs-monoidal}\footnote{The use of ``colax''  refers to the direction of the 2-cell, namely from $F(\nabla_A)$ to $\psi_{A,A}\circ \nabla_{FA}$. Note also that in \cite{FritzGCT23} colax gs-monoidal functors were called colax cartesian.}.
			
	\end{mydefinition}

\begin{myexample}
	Let $\mC$ be a locally small oplax cartesian category. Then, for every object $A$ of $\mC$, the representable functor $\mC(A,-):\mC\to \mathbf{PreOrd}$ has a canonical colax gs-monoidal structure given by 
	\[\psi_{X,Y}: \mC(A,X)\times \mC(A, Y)\to \mC(A,X\otimes Y)\]
	sending $f:A\to X$ and $g:A\to Y$ to $(f\otimes g)\circ \nabla_A$, and 
	\[\psi_0:I\to \mC(A,I)\]
	sending the unique element of $I$ to $!_A$. We refer to \cite[Thm. 6.3]{FritzGCT23} for details.
\end{myexample}

\subsection{An enriched taxonomy of Kleisli categories}\label{subsection:enriched taxonomy}
	\begin{mydefinition}
A symmetric monoidal monad $(T,\mu,\eta, c,u)$ on a preorder-enriched gs-monoidal category is said to be a \textbf{colax gs-monoidal monad} if $T$ is a 
colax gs-monoidal functor.
\end{mydefinition}
\begin{myproposition}\label{prop: Kleisli colax cartesian}
	Let $(T,\mu, \eta,c, u)$ be a colax gs-monoidal monad on a preorder-enriched gs-monoidal category $\mC$. 
	If $\mC$ is oplax cartesian then so is the Kleisli category $\mC_T$.
\end{myproposition}
\begin{proof}
Recall that $\mC_T$ is gs-monoidal thanks to Proposition \ref{prop: Kleisli FGTC} and inherits the preorder-enrichment from $\mC$. We now prove that it is oplax cartesian.

	We first show that for every arrow $f:X\to Y$ in $\mC_T$, i.e. $f:X\to TY$ in $\mC$, it holds that $!^\sharp_Y \circ^\sharp f\le !^\sharp_X$ in $\mC_T$, i.e. that 
	\[T(!_Y) \circ f\le \eta_I\circ !_X\]
	 in $\mC$.
	Applying first the assumption that $T$ is colax affine (i.e. that $T(!_Y)\le \eta_I\circ !_{T(Y)}$) and then  the fact that $\mC$ is oplax cartesian (in particular, that $!_{T(Y)}\circ f\leq\,  !_{X}$), we obtain
	\[T(!_Y)\circ f\le \eta_I\circ !_{T(Y)}\circ f\leq \eta_I\circ !_{X}.\]

	It remains to prove that for every $f:X\to Y$ in $\mC_T$ it holds that $\nabla^\sharp_Y\circ^\sharp f\le  (f\otimes^{\sharp} f)\circ^\sharp\nabla^\sharp_X$ in $\mC_T$, i.e. that 
	\[T(\nabla_Y)\circ f\le  (\mu_{Y\otimes Y} \circ T(c_{Y,Y})\circ T(f\otimes f))\circ (\eta_{X\otimes X}\circ \nabla_X)\]
	in $\mC$. Notice that, using the naturality of $\eta$ and $\mu$, the previous inequality happens to be equivalent to
\[T(\nabla_Y)\circ f\le  c_{Y,Y}\circ (f\otimes f)\circ  \nabla_X\]	
	in $\mC$. But this can be easily derived as follows
	\[
	T(\nabla_Y)\circ f  \le c_{Y,Y}\circ \nabla_{T(Y)}\circ f \le c_{Y,Y}\circ (f\otimes f) \circ \nabla_X \]
	by first using the fact that $T$ is colax relevant (i.e. $T(\nabla_Y) \le c_{Y,Y}\circ \nabla_{T(Y)}$), and then the oplax cartesianity of $\mC$ (in particular, that $\nabla_{T(Y)}\circ f \le  (f\otimes f) \circ \nabla_X$).
\end{proof}

\begin{myexample}
\label{preset}
Given a preorder $(X, \leq)$, the Hoare preorder on its subsets is given by $U \leq_d V$ if 
for any $x \in U$ there exists $y \in V$ such that $x \leq y$. Or, equivalently,
if $\downarrow U \subseteq \downarrow V$, for 
$\downarrow U = \{x \in X \mid \exists u \in U. x \leq u\}$
the \emph{downward-closure} of $X$.

Consider now the finite subset functor $\mathcal{P}$ on $\mathbf{Set}$,
and recall that the sub-category $\mathbf{PreOrd}$ is preorder-enriched
(see Example~\ref{preord}). 
Now, $\mathcal{P}$ can be extended to a preorder-enriched functor
$\mathbf{PreOrd} \to \mathbf{PreOrd}$
with the preorder $\leq_d$ on subsets given above.
Indeed, if $f \leq g: (X, \leq) \to (Y, \leq)$, then
$\mathcal{P}(f) \leq \mathcal{P}(g)$
since $\mathcal{P}(f)(U) = \bigcup_{x \in U} f(x)
\leq_d \bigcup_{x \in U} g(x) = \mathcal{P}(g)(U)$
for all $U \subseteq X$.

%
%
In fact, $\mathcal{P}$ is a colax gs-monoidal monad and it becomes affine or relevant
if one restricts respectively to subsets with  at least or at most one element,
as done in Remark \ref{power}. 
\end{myexample}

%
%

\begin{myremark}
Consider now the full poset-enriched sub-category $\mathbf{PO}$ of $\mathbf{PreOrd}$ of posets.
The functor $\mathcal{P}^\downarrow:\mathbf{PO}\to \mathbf{PO}$ assigns to 
each poset $(X, \leq)$ its (possibly infinite) downward-closed subsets: the preorder $\leq_d$ in Example~\ref{preset} 
is a partial order, since it coincides with subset inclusion.
As for functions, if  $f: (X, \leq) \to (Y, \leq)$, then $\mathcal{P}^\downarrow(f)$ sends a downward-closed  
subset $U$ of $X$ into the  downward-closed subset $\downarrow \bigcup_{x \in U} f(x)$ of $Y$.

It is easy to check that $\mathcal{P}^\downarrow$ is a poset-enriched functor.
It is a colax gs-monoidal monad, 
with essentially the same structure arrows as $\mathcal{P}$, e.g. 
$\eta_X: (X, \leq) \to (\mathcal{P}^{\downarrow}(X), \subseteq)$ sends $x\in X$ to $\downarrow \{x\}$.
This monad becomes affine when restricted to non-empty downward-closed subsets.

The monad $\mathcal{P}^\downarrow$ coincides with the \emph{lower subset monad} in~\cite{banaschewski1991projective}.
Also, as shown in~\cite{banaschewski1991projective}, the algebras of $\mathcal{P}^\downarrow$ are posets with arbitrary sups, 
and the free algebras are supercoherent posets, i.e. posets with arbitrary suprema and where every element 
is a union of supercompact elements. 
Similar considerations can be made for the category $\mathbf{PO}_\bot$ of pointed posets~\cite{kozen2013kleene}.
\end{myremark}

\begin{myexample}
\label{laxSpecial}
	Let $\mathcal{C}$ be a preorder-enriched gs-monoidal category and $M\in\mathcal{C}$ equipped with a monoid
	structure that is \emph{lax}
	special and \emph{lax} connected, i.e. 
	such that it holds
	\ctikzfig{actionmonadexample}
	Then the \textit{action monad} $(-)\otimes M:\mathcal{C}\to \mathcal{C}$ is colax gs-monoidal with 
	the same structure maps as those in Example~\ref{CSMonad}.

	Looking at Example~\ref{example: frobenius span} and Example~\ref{example: pspan e' oplax},
	$(\mathbf{PSpan}(\mathcal{A}),\times,1)$ 
	is an oplax bicartesian category 
	such that each object is special and lax connected, 
	yet not connected.
\end{myexample}

\begin{myexample}
	Given a preorder $(X, \leq)$, we define the \emph{upward-closure} of a subset $ U \subseteq X$ as $\uparrow U = \{x \in X \mid \exists u \in U.\, u \leq x\}$. 
	Consider also a semiring $(M,\oplus, \odot, 0, 1)$ and the functor $\mathcal{M}:\mathbf{Set}\to \mathbf{Set}$ discussed in Example~\ref{example: monad semiring}.
	If $M$ is preorder-enriched and $0$ is a minimal element for such preorder, then $\mathcal{M}(X)$ can be equipped with a preorder given as $h\le_u k: X\to M$ whenever
	for every upward-closed subset $U\subseteq X$ 
	\[\underset{x\in U}{\bigoplus}h(x)\le \underset{x\in U}{\bigoplus}k(x)\]
	 We can then extend $\mathcal{M}$ to a preorder-enriched functor $\mathcal{M}:\mathbf{PreOrd}\to \mathbf{PreOrd}$.
	 Let $f:(X, \leq) \to (Y, \leq)$ be a monotone function and note that for every
	 upward-closed subset $V\subseteq Y$ it holds that $f^{-1}(V)$ is upward-closed.
	 To prove that $\tilde{f}$ is monotone it suffices to see that for $h \leq_u k$ and $V$ upward-closed we have
	 \[ \underset{y\in V}{\bigoplus}\underset{x\in f^{-1}(y)}{\bigoplus}h(x) = \underset{x\in f^{-1}(V)}{\bigoplus}h(x) \le\underset{x\in f^{-1}(V)}{\bigoplus}k(x) = \underset{y\in V}{\bigoplus}\underset{x\in f^{-1}(y)}{\bigoplus}k(x) \]
	To prove that it is preorder-enriched we need to show that if
	$f\le g:X\to Y$ then for every $h:X\to M$ it holds $\tilde{f}(h)\le_u \tilde{g}(h): Y \to M$. Indeed,
	for every $V\subseteq Y$ upward-closed, if $y\in V$ and $x\in f^{-1}(y)$ then $y\le g(x)$ and $g(x)\in V$. Hence
	\[ \underset{y\in V}{\bigoplus}\underset{x\in f^{-1}(y)}{\bigoplus}h(x)\le  \underset{y\in V}{\bigoplus}\underset{x\in g^{-1}(y)}{\bigoplus}h(x) \]
	since $0$ is a minimal element for the preorder on $\mathcal{M}$ and thus $a \leq a \oplus b$ always holds. 
	Moreover, $\mathcal{M}$ is a symmetric monoidal monad with the usual 
	structure arrows
	 \begin{itemize}
		\item[-] \(\eta_X(x_0)(x)=\begin{cases}
			1\ \text{if}\ x=x_0\\
			0\ \text{otherwise}
		\end{cases}\)
		\item[-] $\mu_X(\lambda)(x)=\underset{h\in\mathcal{M}(X)}{\bigoplus}\lambda(h)\odot h(x)$
		\end{itemize}
	 Recall now Example~\ref{sub}.
	 The monad $\mathcal{M}$ just defined becomes colax affine if one restricts $\mathcal{M}(X)$ to functions with finite support and sub-normalised; it is affine if one further restricts to normalised functions.
	 Moreover, it is colax relevant if one restricts $\mathcal{M}(X)$ to sub-idempotent functions 
	 (thanks to 0 being minimal we may drop the requirement that the support is at most one for $\mathcal{M}^s_e(X)$ in Example~\ref{sub}); 
	 it is relevant if one further restricts to idempotent functions  with support at most one. 
		If we consider the semiring of positive real numbers and restrict $\mathcal{M}(X)$ to normalised functions, then the above monad can be viewed as a 
	 discrete version of the preordered Kantorovich monad, as defined in~\cite{fritz2020stochastic}.
\end{myexample}

\begin{myexample}
Let us follow up on the previous example and assume that $M$ is a
\emph{dioid}, i.e. that the addition is idempotent.
It is well-known that the canonical preorder $\leq_M$
is now a partial order, and it can be equivalently defined as $a \leq_M b$ if $a \oplus b = b$.
In fact, $\leq_M$ is a join-semilattice, since $a \vee b = a \oplus b$ and $\bot = 0$.
Let us further assume that such a partial order is actually total. Then, the 
preorder $\leq_u$ can be described as
$h \leq_u k$ if for any $x \in supp(h)$ there exists $y \in supp(k)$ such that 
$x \leq y$ and $h(x) \leq_M k(y)$.
For $\mathcal{M}$ the Boolean semiring, we recover the Hoare preorder
of Example \ref{preset}.
\end{myexample}

\subsection{Some properties of maps}
We first recall the notions of \emph{adjoint} and \emph{map}~\cite{CARBONI198711}.
	
\begin{mydefinition}
	Let $\mC$ be a poset-enriched category. An arrow $f:X\to Y$ is left adjoint to an arrow $f^*:Y\to X$ if
\ctikzfig{adjoints}
Equivalently, $f^*$ is right adjoint to $f$. In symbols $f\dashv f^*$.
\end{mydefinition}

\begin{mydefinition}
		Let $\mC$ be a posetal oplax cartesian category. An arrow $\freccia{X}{f}{Y}$ is a \textbf{map} if it has a right adjoint $f\dashv f^*$.
	\end{mydefinition}
	
We denote by $\mathbf{Map}(\mC)$ the poset-enriched sub-category of $\mC$ 
 whose arrows are maps.

\begin{mylemma}
Let $\mC$ be a posetal oplax cartesian category. Then every map is $\mC$-functional and $\mC$-total.
\end{mylemma}
\begin{proof}
By definition of oplax cartesianity, we need to show that maps  $\freccia{X}{f}{Y}$ satisfy
\ctikzfig{map_are_tot_and_fun}
To show that maps are $\mC$-total it suffices to note that, by definition of left adjoint, we have 
\ctikzfig{proof_map_are_tot}
where the last inequality follows by the second axiom of oplax cartesian categories (applied to $f^*$).
Similarly, we can prove that maps are $\mC$-functional as follows
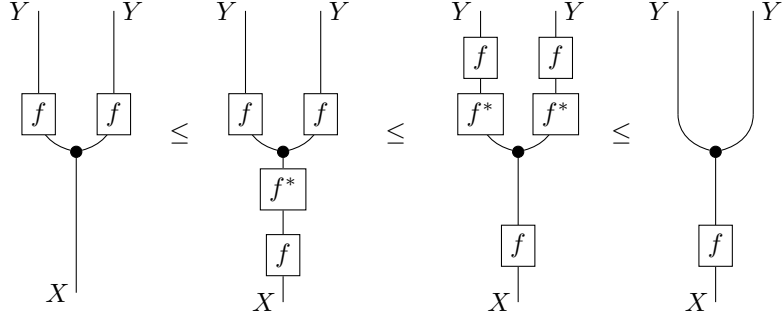
\ctikzfig{proof_map_are_fun} 
where the first and the last inequalities follow from the fact that $f\dashv f^*$, and the second one follows from the first inequality in the definition of oplax cartesian category.
\end{proof}

\begin{myremark}\label{rem_map_sub_cat_of_TFun}
An arrow that is both $\mC$-total and $\mC$-functional is not
necessarily a map. In general,  $\mathbf{Map}(\mC)$ is just a poset-enriched monoidal sub-category of $\mC$-$\mathbf{TFun}$.
\end{myremark}

\subsection{Cartesian bicategories through a gs-monoidal lens}
Now we recall the notion of \emph{cartesian bicategory}~\cite[Def. 1.2]{CARBONI198711}, presenting it in terms of oplax cartesian category.
\begin{mydefinition}\label{def_cartesian_bicategory}
	A \textbf{cartesian bicategory} $\mC$ is a posetal oplax cartesian category such that
 for every object $X$, the arrows $\nabla_X : X \to X \otimes X$ and $!_X : X \to I$ are maps, i.e. there are arrows $\Delta_X : X \otimes X \to X $ and $\cobang_X: I \to X$ such that 
 \ctikzfig{cart_bicat_def}
\ctikzfig{cart_bicat_def_2}
\end{mydefinition}

These inequalities act on orthogonal components, and thus hold for a large number of constructions. 
As partly shown later, they hold for bicategories of spans and cospans and coproduct on 
$\mathbf{Set}$, sometimes with an equality~\cite{Bruni2003}.

The overloaded choice of the symbols $\Delta_X$ and $\cobang_X$ is by no means by chance. 

\begin{myproposition}\label{prop: every cartesian bicategory is oplax cocartesian}
	Every cartesian bicategory $\mC$ is oplax cocartesian.
\end{myproposition}
	\begin{proof}
	First, note that $\mC$ is a cogs-monoidal category. Indeed, all the equations that $\Delta_X$ and $\cobang_X$
	have to satisfy follow from the fact that the right adjoint is unique and from the right adjoint inequalities 
	in Definition \ref{def_cartesian_bicategory}, see also \cite[Rem. 1.3]{CARBONI198711}. The two inequalities in the definition of oplax cocartesian category can be proved as follows: 
	for every arrow $f:X\to Y$
\ctikzfig{oplax_cocart_prop_1}
where the first and the last inequality follow by the properties of right adjoints, while the second one follows from oplax cartesianity. Similarly, for every arrow $f:X\to Y$
\ctikzfig{oplax_cocart_prop_2}\vspace{-1.2\baselineskip}\qedhere
	\end{proof}

We sum up the relationship between bicategories and the gs-monoidal framework.

\begin{mycorollary}
Cartesian bicategories correspond exactly to posetal oplax bicartesian categories 
that are lax special and lax connected and satisfy
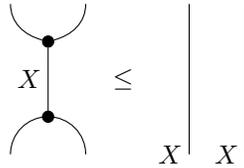
\ctikzfig{cart_bicat_mezzo}
\end{mycorollary}
\begin{proof}
The proposition above implies that a cartesian bicategory
is an oplax bicartesian category, see Definition~\ref{def oplax cartesian cat}.
Now, Remark~\ref{laxbialg} ensures that 
an oplax bicartesian category is also lax algebraic, hence 
the rightmost axiom on the bottom of Definition~\ref{def_cartesian_bicategory}
holds. Then, being lax special and lax connected, see Example~\ref{laxSpecial},
 the additional requirement ensures the correspondence with the remaining inequalities
holding for cartesian bicategories.
\end{proof}

\begin{mydefinition}\label{dfn: bicategory of bialgebras}
A \textbf{bicategory of bialgebras} is a cartesian bicategory 
that is also a bialgebraic category.
\end{mydefinition}

As before, a key difference are the additional laws for their interaction~\cite{BruniGM02}.

\begin{mydefinition}\label{dfn: bicategory of relations}
A \textbf{bicategory of relations} is a cartesian bicategory 
that is also a Frobenius category.
\end{mydefinition}

\begin{myexample}
As expected, $(\mathbf{PRel}(\mathcal{A}),\times,1)$ is a bicategory of relations and $(\mathbf{PRel}(\mathcal{A}),+,0)$ a bicategory of bialgebras.
\end{myexample}

\begin{myremark}
The notion of bicategory of relations is equivalent to the notion of \emph{unitary pretabular allegory} in the sense of Freyd and Scedrov~\cite{freyd1990categories}. Moreover, a bicategory of relations $\mC$ happens to be biequivalent to the regular category $\mathbf{Map}(\mC)$ under the further assumption of being \emph{functionally complete}, as shown in \cite[Th. 3.5]{CARBONI198711}. We refer to \cite{Fong19} for more details.
\end{myremark}

%
%

\begin{myremark}
One of the useful consequences of the Frobenius law is that, as it happens in $\mathbf{Rel}$, in every cartesian bicategory of relations, if $f\leq g$ and both are maps then they are equal~\cite{CARBONI198711}, i.e. the order in $\mathbf{Map}(\mC)$  is discrete.
\end{myremark}

\begin{myremark}\label{rem_map_prod}
We recall from \cite{CARBONI198711} that the category $\mathbf{Map}(\mC)$ of a cartesian bicategory is actually cartesian. Moreover, combining this fact with Corollary~\ref{cor_TFun_is_cartesian} it follows that   $\mathbf{Map}(\mC)$ is a cartesian sub-category of  $\mC$-\textbf{TFun}. However, the cartesian structure is not enough to conclude that every $\mC$-functional and $\mC$-total arrow is a map. Indeed, this happens when $\mC$ is a bicategory of relations~\cite[Lem. 2.5]{CARBONI198711}.
\end{myremark}

\begin{myproposition}
	Let $(T,\mu, \eta,c, u)$ be a colax gs-monoidal monad on a poset-enriched gs-monoidal category $\mC$. 
	If $\mC$ is a cartesian bicategory then so is the Kleisli category $\mC_T$.
\end{myproposition}
\begin{proof}

By Proposition \ref{prop: Kleisli colax cartesian}, we have that $\mC_T$ is oplax cartesian. Moreover, by defining $\Delta^{\sharp}_X : X \otimes X \to X $ and $\cobang^{\sharp}_X: I \to X$ in $\mC_T$ as the arrows $\Delta^{\sharp}_X:=\eta_X\circ \Delta_X$ and $\cobang^{\sharp}_X:= \eta_X\circ\cobang_X$ of $\mC$, we obtain right adjoints
to $\nabla^\sharp$ and $!^\sharp$ since the functor $\mathcal{K}:\mC\to \mC_T$ is poset-enriched. \end{proof}

\begin{mycorollary}
	Let $(T,\mu, \eta,c, u)$ be a colax gs-monoidal monad on a poset-enriched gs-monoidal category $\mC$. 
	If $\mC$ is a bicategory of bialgebras (or a bicategory of relations) then so is the Kleisli category $\mC_T$.
\end{mycorollary}
\begin{proof}
It follows from the fact that the functor $\mathcal{K}: \mC \to \mC_T$ preserves the equalities of arrows obtained through composition and products of structural maps, see Remark \ref{rmk: funtore C ---> C_T}.	
\end{proof}

\section{Conclusions and further work}
\label{sec:conclusions}
The aim of the paper has been twofold. On the one side, we strove for putting some order in the  array of those variants of symmetric monoidal categories, possibly poset-enriched,  
that have been proposed with computational and graphical aims in recent years. We hope that our streamlined presentation in terms of gs-monoidal categories could be beneficial to the 
interested community.
On the other side, we presented a series of results concerning the gs-monoidal and oplax cartesian structures of Kleisli categories, putting some order also on 
those variants of enrichments of commutative monads proposed in the literature.

Future threads of investigations include the extension of the taxonomy towards traced gs-monoidal categories~\cite{Joyal_tracedcategories,CorradiniGadducci99b}, in order to account for systems with feedback, and the lively area investigating alternative notions of Markov categories~\cite{di2025partial,LavoreR23,FritzGPT23}
 and affine monads~\cite{Jacobs16,FritzGPT23}, aimed  at distilling a categorical presentation of probability theory, see e.g.~\cite{Jacobs18}.
  And finally, the exploration of completeness theorems for functorial semantics, see 
 e.g.~\cite{FritzGCT23}.


\section*{Acknowledgements}
This research was partly funded by the Advanced Research + Invention Agency (ARIA) Safeguarded AI Programme
and by the EU through the MSCA SE project QCOMICAL (Grant Agreement 101182520).
The authors are indebted to Nathanael Arkor, Tobias Fritz, Bart Jacobs, and Paolo Perrone for their insightful
comments on a draft of this paper.

\bibliographystyle{alphaurl}
\bibliography{references}


\appendix
\section{Basic notions of category theory}
	\subsection{Regular and extensive categories}
	\begin{mydefinition}\label{def:regular category}
		A category with finite limits $\mC$ is called \textbf{regular} if it has coequalisers of kernel pairs and regular epis are stable under pullbacks.
	\end{mydefinition}
	
	The following is a well known characterisation of regular categories, see \cite{Granregular}.
	
	\begin{mytheorem}\label{thm:regular char}
		Let $\mC$ be a category with finite limits. Then $\mC$ is regular if and only if  every arrow $f:X\to Y$ factors as $f=m\circ e$ where $e:X\to I$ is a regular epi and $m:I\to Y$ is a mono, 
		and such factorisation is stable under pullbacks.
		\end{mytheorem}

	\begin{mydefinition}\label{def:extensive category}
		A category $\mC$ with finite coproducts is \textbf{extensive} if for every pair of objects $X_1,X_2$ the canonical functor
		\[\mC/X_1\times \mC/X_2\to \mC/(X_1+X_2)\] is an equivalence of categories.
	\end{mydefinition}

\subsection{Lax monoidal functors}
\label{sec:lax_app}
This section recalls the details of the definition of lax monoidal functor, see e.g.~\cite{aguiar2010}.
Throughout, $\mC$ and $\mD$ are symmetric monoidal categories with tensor functor $\otimes$ and monoidal unit $I$, and we assume that 
$\otimes$ strictly associates without loss of generality in order to keep the diagrams simple.
Left and right unitors are denoted by $\lambda$ and $\rho$, respectively\footnote{Strict unitality could be assumed, but it would make some diagrams potentially confusing.}, and braidings by $\gamma$.

\begin{mydefinition}\label{def:lax monoidal functor}
A \textbf{lax monoidal} functor is a triple $(F,\psi_0,\psi)$ where $\freccia{\mC}{F}{\mD}$ is a functor, equipped with a natural transformation 
\[\freccia{\otimes \circ \, (F\times F)}{\psi}{F\circ \otimes}\]
and an arrow $\freccia{I}{\psi_0}{F(I)}$ such that the associativity diagram
\[\begin{tikzcd}[column sep=3ex]
	{F(A)\otimes F(B)\otimes F(C)} &&& {F(A)\otimes F(B\otimes C)} \\
	\\
	{F(A\otimes B)\otimes F(C)} &&& {F(A\otimes B\otimes C)}
	\arrow["{\id\otimes \,\psi_{B,C}}", from=1-1, to=1-4]
	\arrow["{\psi_{A,B}\otimes {}\id}"', from=1-1, to=3-1]
	\arrow["{\psi_{A\otimes B,C}}"', from=3-1, to=3-4]
	\arrow["{\psi_{A,B\otimes C}}", from=1-4, to=3-4]
\end{tikzcd}\]
and the unitality diagrams commute
\begin{equation}
\label{eq:lax_monoidal_unitality}
\begin{tikzcd}[column sep=1ex]
	{I\otimes F(A)} && F(A) && {F(A)\otimes  I} && F(A) \\
	\\
	{F(I)\otimes F(A)} && {F(I\otimes A)} && {F(A)\otimes F(I)} && {F(A\otimes  I).}
	\arrow["{\psi_{I,A}}"', from=3-1, to=3-3]
	\arrow["{F(\lambda_A)}", from=1-3, to=3-3]
	\arrow["{\psi_0\otimes {}\id}"', from=1-1, to=3-1]
	\arrow["{\lambda_{FA}}"', from=1-3, to=1-1]
	\arrow["{\rho_{FA}}"', from=1-7, to=1-5]
	\arrow["{\id\otimes \psi_0}"', from=1-5, to=3-5]
	\arrow["{\psi_{A,I}}"', from=3-5, to=3-7]
	\arrow["{F(\rho_A)}", from=1-7, to=3-7]
\end{tikzcd}
\end{equation}

$F$ is called \textbf{lax symmetric monoidal} if also the following diagram commutes

\begin{center}
\begin{tikzcd}
	{F(A)\otimes F(B)} && {F(B)\otimes F(A)} \\
	{F(A\otimes B)} && {F(B\otimes A)}
	\arrow["{\gamma^{\mathcal{D}}_{FA,FB}}", from=1-1, to=1-3]
	\arrow["{\psi_{A,B}}", from=1-1, to=2-1]
	\arrow["{\psi_ {B,A}}"', from=1-3, to=2-3]
	\arrow["{F(\gamma^{\mathcal{C}}_{A,B})}"', from=2-1, to=2-3]
\end{tikzcd}
\end{center}
\end{mydefinition}

For example, if $\mC$ is the terminal monoidal category with only one object $I$ and $\id_I$ as the only arrow, then $F$ is simply a monoid in $\mD$.
We do not spell out the following dual version in full details.

We also have \textbf{strong symmetric monoidal functors}, which are lax symmetric monoidal functors with invertible structure arrows; and \textbf{strict symmetric monoidal functors}, in which the structure arrows are identities. 

\begin{mydefinition}
\label{laxtras}
	A \textbf{monoidal transformation} between lax monoidal functors $\freccia{\!\!(F,\psi_0,\psi)\!\!}{\epsilon}{(F',\psi'_0,\psi')}:\mC\to\mD$ is a family of arrows $\epsilon_X:F(X)\to F'(X)$, for $X\in\mC$, satisfying
\begin{equation}\label{eq:monoidal transformation}
\begin{tikzcd}[column sep=tiny]
	{F(X)\otimes F(Y)} && {F'(X)\otimes F'(Y)} && I && {F(I)} \\
	{F(X\otimes Y)} && {F'(X\otimes Y)} &&& {F'(I)}
	\arrow["{\epsilon_X\otimes\epsilon_Y}", from=1-1, to=1-3]
	\arrow["\psi"', from=1-1, to=2-1]
	\arrow["{\psi'}", from=1-3, to=2-3]
	\arrow["{\psi_0}", from=1-5, to=1-7]
	\arrow["{\psi'_0}"', from=1-5, to=2-6]
	\arrow["{\epsilon_I}", from=1-7, to=2-6]
	\arrow["{\epsilon_{X\otimes Y}}"', from=2-1, to=2-3]
\end{tikzcd}\end{equation}
When $\epsilon$ is also a natural transformation between the underlying functors $F,F'$ it is called \textbf{monoidal natural transformation}.
\end{mydefinition}
A monoid and comonoid structure on an object in a symmetric monoidal category often interact in a nice way, either such that they form a \emph{bimonoid} or a \emph{Frobenius monoid} (and sometimes both).

\section{More on cogs-monoidal categories}\label{section: appendix dual results}
\begin{mydefinition}\label{def cogs monoidal functor}
	For cogs-monoidal categories $\mC$ and $\mD$, a functor $\freccia{\mC}{F}{\mD}$ equipped with a lax symmetric monoidal structure
   \[
				\freccia{\otimes \circ \, (F\times F)}{\psi}{F\circ \otimes}, \qquad \freccia{I}{\psi_0}{F(I)} 
			\]
is \textbf{coaffine}  if the following diagram commutes for all $X$ in $\mC$
\begin{equation}\label{diagram: lax coaffine}
\begin{tikzcd}[column sep=tiny]
	F(X) && {F(I)} \\
& I
\arrow["{F(\cobang_X)}"', from=1-3, to=1-1]
\arrow["{\cobang_{FX}}", from=2-2, to=1-1]
\arrow["{\psi_0}"', from=2-2, to=1-3]
\end{tikzcd}
\end{equation}
and it is \textbf{corelevant} if the following diagram commutes for all $X$ in $\mC$

\begin{equation}\label{diagram: lax corelevant}
\begin{tikzcd}[column sep=tiny]
	{F(X)} && {F(X\otimes X)} \\
	& {F(X)\otimes F(X)}
	\arrow["{F(\Delta_X)}"', from=1-3, to=1-1]
	\arrow["{\Delta_{FX}}", from=2-2, to=1-1]
	\arrow["{\psi_{X,X}}"', from=2-2, to=1-3]
\end{tikzcd}
\end{equation}
A functor which is both coaffine and corelevant is \textbf{cogs-monoidal}.
\end{mydefinition}

\begin{myproposition}\label{prop: Kleisli of coaffine monads}
    Let $(T,\mu, \eta,c, u)$ be a symmetric monoidal monad on a cogs-monoidal category $\mC$. If $\mC$ has coprojections then so does the Kleisli category $\mC_T$.
	\begin{proof}
If $I$ is initial then it is so also in $\mC_T$ since the functor $\mC\to \mC_T$ preserves colimits.
    \end{proof}
\end{myproposition}
\begin{myremark}
	In the above proposition it is not necessary to assume the monad to be coaffine. Indeed,
	if $I$ is initial then any symmetric monoidal monad is trivially coaffine.
\end{myremark}

\begin{myproposition}\label{prop: Kleisli of corelevant monads}
	Let $(T,\mu, \eta,c, u)$ be a corelevant monad on a cogs-monoidal category $\mC$. If $\mC$ has codiagonals then so does the Kleisli category $\mC_T$.
\begin{proof}
	We have to prove that for every $f:X\to Y$ in $\mC_T$, which corresponds to an arrow $f:X\to T(Y)$ in $\mC$, it holds that $f\circ^\sharp\Delta^\sharp_X= \Delta^\sharp_Y \circ^\sharp (f\otimes_T f)$ in $\mC_T$. 
	Indeed
	\begin{align}
		f\circ^\sharp \Delta^\sharp_X &= f\circ  \Delta_X \tag*{}\\
		&= \Delta_{T(Y)}\circ (f\otimes f) \tag{$\mC$ codiagonals}\\
		&= T(\Delta_Y) \circ c_{Y,Y}\circ (f\otimes f) \tag{$T$ corelevant}\\
		&= \Delta^\sharp_Y\circ^\sharp (f\otimes_T f) \tag*{\qedhere}
	\end{align}
\end{proof}
\end{myproposition}


\begin{mydefinition}\label{def: oplax cocartesian functor}
	Let $\mC$ and $\mD$
	be oplax cocartesian categories and  
	$\freccia{\mC}{F}{\mD}$ a preorder-enriched lax symmetric monoidal functor with structure arrows $\psi, \psi_0$.
	The $F$ is called
		\textbf{colax coaffine} if the following inequality holds 
\[
\begin{tikzcd}
{F(X)} && {F(I)} \\
& I
\arrow[""{name=0, anchor=center, inner sep=0}, "{F(\cobang_X)}"', from=1-3, to=1-1]
\arrow["{\cobang_{F(X)}}", from=2-2, to=1-1]
\arrow["{\psi_0}"', from=2-2, to=1-3]
\arrow["\le"{anchor=center}, draw=none, from=2-2, to=0]
\end{tikzcd}
\]
and it is called \textbf{colax corelevant}  if the following inequality holds
\[
\begin{tikzcd}
{F(X)} && {F(X\otimes X)} \\
& {F(X)\otimes F(X)}
\arrow[""{name=0, anchor=center, inner sep=0}, "{F(\Delta_X)}"', from=1-3, to=1-1]
\arrow["{\Delta_{F(X)}}", from=2-2, to=1-1]
\arrow["{\psi_{X,X}}"', from=2-2, to=1-3]
\arrow["\le"{description}, draw=none, from=2-2, to=0]
\end{tikzcd}\]
If $F$ is both colax affine and colax relevant it is called \textbf{colax cogs-monoidal}.
\end{mydefinition}

\begin{myproposition}\label{prop: Kleisli colax cocartesian}
	Let $(T,\mu, \eta,c, u)$ be a colax cogs-monoidal monad on a preorder-enriched cogs-monoidal category $\mC$. 
	If $\mC$ is oplax cocartesian then so is the Kleisli category $\mC_T$.
\end{myproposition}
\begin{proof}
	$\mC_T$ is cogs-monoidal thanks to Proposition \ref{prop:Kleisli cogs} and inherits the preorder-enrichment from $\mC$. We now prove that it is oplax cocartesian.

	We first show that for every arrow $f:I\to T(X)$ in $\mC_T$ it holds that $f\le \cobang^\sharp_X$.
	Since $\mC$ is oplax cocartesian we have that $f\le \cobang_{T(X)}$, and because $T$ is colax coaffine it implies
	that $\cobang_{T(X)}\le T(\cobang_X)\circ \eta_I= \cobang^\sharp_X$.

	It remains to prove that for every $f:X\to Y$ in $\mC_T$, which corresponds to an arrow $f:X\to T(Y)$ in $\mC$, it holds that $f\circ^\sharp \Delta^\sharp_X\le \Delta^\sharp_Y \circ^\sharp (f\otimes_T f)$ in $\mC_T$. Indeed
	\begin{align}
		f\circ^\sharp \Delta^\sharp_X &= f \circ\Delta_X \tag*{}\\
		&\le \Delta_{T(Y)}\circ (f\otimes f) \tag{$\mC$ oplax cocartesian}\\
		&\le T(\Delta_Y)\circ c_{Y,Y}\circ (f\otimes f) \tag{$T$ colax corelevant}\\
		&= \Delta^\sharp_Y\circ^\sharp (f\otimes_T f) \tag*{\qedhere}
	\end{align}
	\end{proof}

\end{document}